\documentclass[11pt]{article}
\usepackage{amsmath,amssymb,amsthm,graphicx,subfigure,float,url}
\usepackage{fullpage}
\usepackage{pdfsync}
\usepackage{stmaryrd}
\usepackage{psfrag}
\usepackage{mathrsfs}
\usepackage{enumitem}

\newtheorem{thrm}{Theorem}  
\newtheorem{prpstn}{Proposition}

\newtheorem{dfntn}{Definition}
\newtheorem{lmm}{Lemma}
\newtheorem{xmpl}{Example}
\newtheorem{rmrk}{Remark}

\renewcommand{\d}{\textnormal{d}}

\date{}

\title{Shape minimization of the dissipated energy in dyadic trees}
\author{Xavier Dubois de La Sabloni\`ere\footnotemark[4]~\footnotemark[1]\and 
Benjamin Mauroy\footnotemark[2]
\and 
Yannick Privat\footnotemark[3]
}

\footnotetext[4]{Sup\'elec, Plateau de Moulon, 91192 Gif-sur-Yvette, France; xavier.duboisdelasabloniere.2010@asso-supelec.org}
\footnotetext[1]{Universit\'e d'Orl\'eans, Laboratoire MAPMO, CNRS, UMR 6628, F\'ed\'eration Denis Poisson, FR 2964, Bat. Math., BP 6759, 45067 Orl\'eans cedex 2, France;}
\footnotetext[2]{Laboratoire MSC, Universit\'e Paris 7, CNRS, 10 rue Alice Domon et L\'eonie Duguet, F-75205 Paris cedex 13, France;
Benjamin.Mauroy@univ-paris-diderot.fr}
\footnotetext[3]{ENS Cachan Bretagne, CNRS, Univ. Rennes 1, IRMAR,
av. Robert Schuman, F-35170 Bruz, France; yannick.privat@bretagne.ens-cachan.fr}

\begin{document}

\maketitle

\begin{abstract}

\par In this paper, we study the role of boundary conditions on the optimal shape of a dyadic tree in which flows a Newtonian fluid. Our optimization problem consists in finding the shape of the tree that minimizes the viscous energy dissipated by the fluid with a constrained volume, under the assumption that the total flow of the fluid is conserved throughout the structure. 
These hypotheses model situations where a fluid is transported from a source towards a 3D domain into which the transport network also spans. Such situations could be encountered in organs like for instance the lungs and the vascular networks.

Two fluid regimes are studied: (i) low flow regime (Poiseuille) in trees with an arbitrary number of generations using a matricial approach and (ii) non linear flow regime (Navier-Stokes, moderate regime with a Reynolds number $100$) in trees of two generations using shape derivatives in an augmented Lagrangian algorithm coupled with a 2D/3D finite elements code to solve Navier-Stokes equations. It relies on the study of a finite dimensional optimization problem in the case (i) and on a standard shape optimization problem in the case (ii).
We show that the behaviours of both regimes are very similar and that the optimal shape is highly dependent on the boundary conditions of the fluid applied at the leaves of the tree. 
\end{abstract}
\paragraph{Acknowledgement:} {The third author is partially supported by the ANR project GAOS ``Geometric Analysis of Optimal Shapes'' .}\\
\maketitle
\newpage
\section*{Introduction and motivations}
Tree structures are very common means to transport a product between two regions of different scales. A fluid acting as a transporter often flows in such structures. The circulation of the fluid in such geometries can dissipate a lot of energy by viscous effects and the question of the optimization of the tree geometry arises. Important applications of this problem exist, like in industry, in the study of river basins, in water treatment, in medicine. Such structures are also often encountered in Biology and are the result of evolution. One can see roughly natural selection as an optimization process which adapts the organisms to their environment. Thus, a major question arises for biological systems: what are they optimized for? We know that an effect of natural selection is to minimize some complex cost functions relatively to a wide range of parameters. Although most of the time unknown, these parameters have probably various influences in term of amplitude, some stronger than others. Thus, if we are able, typically through a modeling work, to give hypotheses on what are the most influential parameters, to isolate them in a model and to determine their role if they were alone, then a simple comparison between the results of the model and the real biological system could indicate whether their role was truly important or not.
In the case of biological networks such as lungs or vascular network, two cost parameters arise naturally: the viscous dissipated energy of the fluid in the network and the volume of the network. Indeed, these organs have to deal with energy dissipation due to air or blood circulation \cite{weibel} and since they span in the geometry they have to feed, they cannot use too much volume. These organs can be subject to dysfunctions which are often consequences of an increase of their hydrodynamical resistance and thus of a loss of efficiency of the geometrical structures as transport systems. Hence, dysfunctions like asthma or heart attacks are often linked to an increase of the hydrodynamic resistance (energy cost spent for the circulation of the fluid) of the structure. Thus, researching the optimal shapes of tree structures could bring useful information. Previous studies have been made on this topic in the past like in \cite{Hess, bejan, tondeur, mauroy-optimalTreeDangerous, bernot} each on particular situations and applications.

In this frame, the goal of this work is to determine the shapes of dyadic trees that would minimize the viscous energy of a Newtonian fluid under a volume constraint on the tree. As written upwards, such structures are good candidates for the modeling of mammals bronchial trees \cite{mauroy-optimalTreeDangerous} and we will focus most particularly on this application. We study this problem for two regimes of flow. We begin with low flow regime (Poiseuille flow) using a matricial formulation of the problem as in \cite{mauryMeunier, mauroy-meunier}. We also study the non linear Navier-Stokes flow at a moderate Reynolds number ($\sim 100$) which is however sufficient to exhibit inertial effects around the bifurcation. In this second case, we use a numerical shape minimization method based on shape derivatives \cite{HP}. We show that for both regime the optimal structure depends on the fluid boundary conditions that are applied at tree root and leaves. Indeed, under the assumption that the total viscous flow is constant throughout the tree, the optimal shape is very different according to the boundary conditions imposed at the leaves: Dirichlet conditions or strictly identical Neumann conditions lead to an optimal tree with particular relationships between the flow in a branch and its diameter; non identical Neumann conditions lead to a degenerated tree reduced to a tube and only one leaf remains accessible to the fluid from the root. Moreover the numerical simulations in the non linear case give precisely the geometry of the bifurcation. We focus here on the 3D case, however it is easy to see, with very few changes in the reasoning, that our results would also hold for the 2D case.

\section{Terminology and notations}

In the following, we will call {\it inlet} of a dyadic tree either the open surface of the root of the tree or the root of the tree itself, depending on the context. Similarly, we will call {\it outlets} of a tree either the open surfaces of its leaves or its leaves themselves, also depending on the context. Mainly, we will refer to the open surfaces if we speak of boundary conditions and to the branches in the other cases.
Note that this terminology does not mean necessarily that fluid is going in the tree through the inlet and out of the tree through the outlet.

We will use the following notations throughout this section:\\
\begin{center}
\begin{tabular}{p{8cm}|p{8cm}}
$[x]$, with $x\in \mathbb{R}$ & the integer part of $x$\\
$B^\top$, where $B$ is a matrix & the transpose of $B$\\
$(\boldsymbol{e_1},\dots,\boldsymbol{e_n})$, with $n\in\mathbb{N}^*$ & the canonical basis of $\mathbb{R}$\\
$\langle\boldsymbol{x},\boldsymbol{y}\rangle$, where $\boldsymbol{x}$ and $\boldsymbol{y}$ are two vectors with same length & the euclidean inner product\\
$\lVert .\rVert$ & the euclidian norm, induced by the inner product $\langle .,.\rangle$\\
$\textnormal{diag}\boldsymbol{u}$, where $\boldsymbol{u}=(u_i)_{i\in \llbracket 1,n\rrbracket}\in \mathbb{R}^{n\times 1}$ & the diagonal matrix $D=(d_{i,j})_{1\leq i,j\leq n}$ such that $d_{i,i}=u_i$ for all $i\in \llbracket 1,n\rrbracket$.
\end{tabular}
\end{center}

In the following, a family of real numbers $(\nu_{\varepsilon})_{\varepsilon \in \mathbb{R}_+^*}$ is assimilated to a sequence since for each $n \in \mathbb{N}^*$, one can choose $\varepsilon = \frac1n$. \\

Moreover, let us define the direct sum of two matrices.
\begin{dfntn}
Let $m$ and $n$ be two nonzero integers. Let $A_1\in \mathbb{R}^{n\times n}$ and $A_2\in \mathbb{R}^{m\times m}$ be two matrices. The direct sum of $A_1$ and $A_2$ is the matrix $M$ in $\mathbb{R}^{n+m,n+m}$ defined by blocks by
$$
M=  A_1\oplus A_2=\left(\begin{array}{c|c}
A_1 & 0 \\
\hline
0 & A_2
\end{array}\right).
$$
\end{dfntn}

\section{Models}\label{The model}
We introduce here all the models used in this article. For the sake of clarity, all proofs of this section have been regrouped in Appendix \ref{app1}.

\subsection{Poiseuille's law}
We consider a viscous incompressible fluid whose dynamic viscosity is $\mu>0$. It flows in a steady and laminar state through a cylindrical rigid pipe whose length is $L>0$ and radius is $R>0$. We impose a \textit{no-slip} condition on the lateral boundary which means that the fluid ``sticks'' to the wall. We refer the inlet to $0$ and the outlet to $1$. Pressures ($p_0>0$ and $p_1>0$) at its openings are supposed to be uniform all over the section. The volumetric flow rate $\Phi$ is chosen positive if the flow goes from section $0$ to section $1$.\\
The fluid behavior is ruled by the Navier-Stokes equations in the cylinder. The solution of these equations with the previous conditions is characterized with velocity profiles that are parabolic on each section and a pressure that is constant on each section and decrease linearly along the axis of the pipe. This type of flow is known to verify Poiseuille's law. This law boils down to a linear relationship between the pressure drop $p_0 - p_1$ and the volumetric flow rate $\Phi$ through the pipe, that is
$$
p_0 - p_1 = r\Phi
$$
where
$$
r = \frac{8 \mu L}{\pi R^4} > 0.
$$\\
Therefore, if the pressure drop $p_0 - p_1$ is positive, the flow goes from section $0$ to section $1$.
\newline

There is an exact correspondence between Poiseuille's law and Ohm's law if we match pressure drop with potential difference and volumetric flow rate with electric current. Thus by analogy, the proportionality constant $r$ is called a hydrodynamic resistance. Finally, since the flow is incompressible, the flow rate is conserved through the pipe.

\subsection{Dyadic trees}

In this section, we will define the different mathematical tools needed in the sequel to manipulate tree structures in which flows a Poiseuille's fluid. In particular, we give a matricial relationship between the flow at the inlet of a dyadic tree and the pressures at the outlets.

We consider the flow of an incompressible and viscous fluid whose dynamic viscosity is $\mu $ through a finite dyadic tree of $N+1$ generations ($N \in \mathbb{N}$, $N \geq 1 $). We recall that a new {\it generation} of this tree occurs when a bifurcation is created. Hence the root branch corresponds to generation 1 and the branches at the leaves of the tree corresponds to generation $N+1$. Consequently, the tree has $2^{N}$ outlets and $2^{N+1} - 1$ branches. 

We also define the {\it level} as a number associated to a generation, equal to 1 for the second generation and increased by 1 at each generation. Therefore, a tree with $N+1$ generations has $N$ levels. 

We do the distinction between the generation and the level in the tree, because it makes the indexation of the different variables of the tree easier. The levels will be denoted in all this section by the index letter $i$.

\par We assume that our tree is composed of connected rigid cylindrical pipes in which the fluid obeys to Poiseuille's law. We call $\Phi>0$ the volumetric flow rate that enters the tree. Since the flow is incompressible, the flow rate is conserved through the pipes and at bifurcations.
\par The sets $\mathcal{B}_{N,i}$ of couples of indexes that locate each branch of a given level $i$ ($i \in \llbracket 1,N \rrbracket$) is
\begin{equation}
\mathcal{B}_{N,i} = \left\{(i,j) \mid j \in \llbracket 1,2^i \rrbracket\right\}.
\end{equation}
Hence, $i$ represents the level and $j$ the position of the branch at this level. Thus, the set $\mathcal{B}_N$ of all indexes locating the pipes for the overall tree is
\begin{equation}
\mathcal{B}_N = \bigcup_{i \in \llbracket 1,N \rrbracket} \mathcal{B}_{N,i}\\
\end{equation}

This set does not include the root branch of the tree. This branch has a radius $R_0>0$, a length $L_0>0$, the pressure at its inlet is $p_0>0$ (pressure at the inlet of the tree) and is $p_1>0$ at its outlet. Thus, its hydrodynamic resistance is $r_0 = \frac{8 \mu L_0}{\pi R_0^4}$. According to Poiseuille's law, one has
$$p_0 - p_1 = r_0\Phi.$$
\par For a pipe whose location in the tree is given by the couple $(i,j) \in \mathcal{B}_N$, we denote $R_{i,j}>0$ its radius and $L_{i,j}>0$ its length. Therefore, the hydrodynamic resistance of this pipe is $r_{i,j} = \displaystyle \frac{8 \mu L_{i,j}}{\pi R_{i,j}^4}$. We use $q_{i,j}$ and $p_{i,j}>0$ to define respectively, the volumetric flow rate through this pipe and the pressure at its outlet. The flow rate $q_{i,j}$ is chosen positive if the fluid in a pipe flows towards the pipes with a higher generation index.

\par We consider now a pipe whose index is $(i,j) \in \displaystyle \bigcup_{i \in \llbracket 1,N-1 \rrbracket} \mathcal{B}_{N,i}$. Since the tree is dyadic, the hydrodynamic resistances of its (two) daughter pipes are $r_{i+1,2j-1}$ and $r_{i+1,2j}$, which belong to the $(i+1)$-th level. Then, we define the reduction ratios $x_{i+1,2j}$ and $x_{i+1,2j+1}$, that represents the change in the geometry of the pipes between the levels $i$ and $i+1$, by
\begin{equation}
x_{i+1,2j-1} = \frac{r_{i,j}}{r_{i+1,2j-1}} = \frac{L_{i,j}}{L_{i+1,2j-1}} \left(\frac{R_{i+1,2j-1}}{R_{i,j}} \right)^4 \textrm{~and~} 
x_{i+1,2j} = \frac{r_{i,j}}{r_{i+1,2j}}= \frac{L_{i,j}}{L_{i+1,2j}} \left(\frac{R_{i+1,2j}}{R_{i,j}} \right)^4.
\end{equation}
\par Moreover, we assume an identical reduction ratio for the radius and the length between a mother branch and its daughter, thus
\begin{equation}
x_{i+1,2j-1} = \left(\frac{R_{i+1,2j-1}}{R_{i,j}} \right)^3  \textrm{~and~} x_{i+1,2j} = \left(\frac{R_{i+1,2j}}{R_{i,j}} \right)^3.
\end{equation}\ \\
\par Figure \ref{treePicture} shows an example of a dyadic tree and of our notations in the special case where $N$ is equal to 2.
\begin{psfrags}
\psfrag{phi}{$\Phi$}
\psfrag{p0}{$p_0$}
\psfrag{p1}{$p_1$}
\psfrag{p11}{$p_{1,1}$}
\psfrag{p12}{$p_{1,2}$}
\psfrag{p21}{$p_{2,1}$}
\psfrag{p22}{$p_{2,2}$}
\psfrag{p23}{$p_{2,3}$}
\psfrag{p24}{$p_{2,4}$}
\psfrag{q11}{$q_{1,1}$}
\psfrag{q12}{$q_{1,2}$}
\psfrag{q21}{$q_{2,1}$}
\psfrag{q22}{$q_{2,2}$}
\psfrag{q23}{$q_{2,3}$}
\psfrag{q24}{$q_{2,4}$}
\psfrag{x11}{\small $x_{1,1}$}
\psfrag{x12}{\small $x_{1,2}$}
\psfrag{x21}{\small $x_{2,1}$}
\psfrag{x22}{\small $x_{2,2}$}
\psfrag{x23}{\small $x_{2,3}$}
\psfrag{x24}{\small $x_{2,4}$}
\psfrag{i_1}{$i=1$}
\psfrag{i_2}{$i=2$}
\psfrag{r0}{$r_0$}
\begin{figure}[!h]
\begin{center}
\includegraphics[height=5cm]{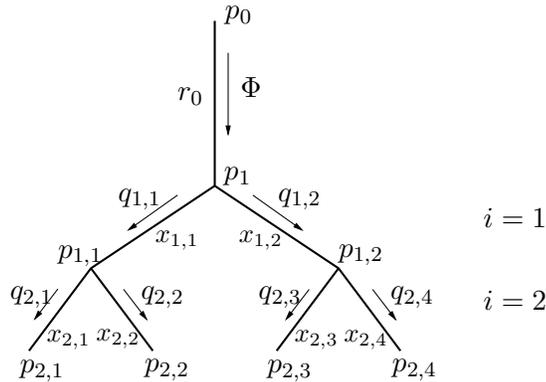}
\caption{An example of dyadic tree with three generations ($N=2$). Note that the flows $q_{i,j}$ can be either positive or negative ($(i,j) \in \mathcal{B}_N$).}
\label{treePicture}
\end{center}
\end{figure}
\end{psfrags}
\par Our goal is to establish a relationship between the $2^{N}$ pressures and the $2^{N}$ volumetric flow rates at the outlets of the tree. To go further, we need to be able to follow the fluid through paths in the tree. Therefore,we define the notions of path and subpath in the tree.\\
\begin{dfntn}Notion of path and subpath.
\begin{enumerate}
\item A path $\Pi_{0 \to (i,j)}$ (with $(i,j) \in \mathcal{B}_N$) is the set of couples of indexes of the $i$ branches needed to link the root branch denoted by $0$ and the branch located by $(i,j)$ in the tree. It includes the branch referred to by $(i,j)$ but not the root branch. More precisely,
\begin{equation}
\Pi_{0 \to (i,j)}=\left\{\left( 1,m_{i}\right),\ldots,\left(i-1,m_2\right),(i,m_1)\right\},
\end{equation}
where $(m_k)_{1\leq k\leq i}$ denotes the sequence of positive integers defined by
$$
\left\{
\begin{array}{l}
m_1=j\\
m_{k+1}=\displaystyle \left[\frac{m_k+1}{2}\right], \ \forall k\in \llbracket 1,i-1\rrbracket .
\end{array}
\right.
$$
\item Let $(i,j) \in \mathcal{B}_N$ and $ s \in \llbracket 0,i \rrbracket $. For a given path $\Pi_{0 \to (i,j)}$, the subpath $\Pi_{0 \to (i,j)}(s)$ is the set of couples of indexes of the $m$ branches needed to link the root branch and a branch located at the $s$-th level following a part of the path $\Pi_{0 \to (i,j)}$. More precisely, the subpath $\Pi_{0 \to (i,j)}(s)$ is the subset of $\Pi_{0 \to (i,j)}$ defined by
\begin{equation}
\Pi_{0 \to (i,j)}(s) = \left \{
\begin{array}{ll}
\left\{\left( 1,m_i\right),\ldots,\left(s,m_{i-s+1}\right)\right\} & \textrm{if }s \geq 1\\
\emptyset & \textrm{if }s = 0.\\
\end{array}
\right.
\end{equation}
\end{enumerate}
\end{dfntn}\ \\
\par We use this definition to do a change of variable. It will be very useful in the second part of this section devoted to the optimization of a given criterion with respect to the geometry of the tree represented by the $2^{N+1}-2$ variables $\{x_{i,j}\}_{(i,j) \in \mathcal{B}_N}$. From now on, we replace the variables $x_{i,j}$ by the new ones $\xi_{i,j}$ defined by
\begin{equation}
\forall (i,j) \in \mathcal{B}_N, \ \xi_{i,j}  = \prod_{(k,l) \in \Pi_{0 \to (i,j)}}x_{k,l}. 
\end{equation}

\par We will impose another constraint on the geometry: lengths and radii of the tree are assumed to decrease as we go along its levels. More precisely,
\begin{equation}\label{cond2}
\forall (i,j) \in \displaystyle \bigcup_{i \in \llbracket 1,N-1 \rrbracket} \mathcal{B}_{N,i}, \ 
\max(\xi_{i+1, 2j-1},\xi_{i+1, 2j}) \le \xi_{i,j}.
\end{equation}

It has to be noticed that the map $\{x_{i,j}\}_{(i,j) \in \mathcal{B}_N}\longmapsto \{\xi_{i,j}\}_{(i,j) \in \mathcal{B}_N}$ is obviously a $C^1$-diffeomorphism on the set of strictly positive real numbers so that it defines a change of variable.\\
Moreover, $\frac{r_0}{\xi_{i,j}}$ represents the hydrodynamic resistance $r_{i,j}$ of the pipe denoted by $(i,j)$. \\
For instance, in the case $N=2$ (see figure \ref{treePicture}), the path linking the inlet to the first outlet is $\Pi_{0 \to (2,1)} = \{(1,1);(2,1)\}$ and the geometric variables of theses branches are
$$
x_{1,1} = \frac{r_0}{r_{1,1}} = \left(\frac{R_{1,1}}{R_0}\right)^3 , \ \xi_{1,1} =x_{1,1} \textrm{, and } x_{2,1} = \frac{r_0}{r_{2,1}} = \left(\frac{R_{2,1}}{R_0}\right)^3  
, \ \xi_{2,1} =x_{1,1}x_{1,2}.
$$

Since the tree is dyadic, it is usual to compute the total volume $V$ of the tree by summing the volume of each cylindrical branch, \textit{i.e.}
\begin{eqnarray}
\textnormal{Volume} & = & \pi R_0^2 L_0 + \sum_{(i,j) \in \mathcal{B}_N} \pi R_{i,j}^2 L_{i,j} = \pi R_0^2 L_0 \left ( 1 + \sum_{(i,j) \in \mathcal{B}_N} \left(\frac{R_{i,j}}{R_0} \right)^2 \frac{L_{i,j}}{L_0} \right ) \nonumber \\
  & = & \pi R_0^2 L_0 \left ( 1 + \sum_{(i,j) \in \mathcal{B}_N} \left(\frac{R_{i,j}}{R_0} \right)^3 \right ) = \pi R_0^2 L_0 \left ( 1 + \sum_{(i,j) \in \mathcal{B}_N} \prod_{(k,l) \in \Pi_{0 \to (i,j)}}x_{k,l} \right ) \nonumber \\
& = & \pi R_0^2 L_0 \left ( 1 + \sum_{(i,j) \in \mathcal{B}_N} \xi_{i,j} \right ).\label{defVol}
\end{eqnarray}

Notice that this volume is an approximation of that of the real tree, since it takes into account only the volumes of the cylinders composing the tree and not the volumes of the bifurcations.\\

Now, let us define
\begin{itemize}
\item the vector $\boldsymbol{p}$ containing all the pressures at the outlet of the tree, \textit{i.e.} 
$$
\boldsymbol{p} = (p_{N,j})_{j \in \llbracket 1,2^{N} \rrbracket}^\top\in (\mathbb{R}_+^*)^{2^{N} \times 1},
$$
\item the vector $\boldsymbol{q}$ containing all the volumetric flow rates at the leaves of the tree, \textit{i.e.} 
$$
\boldsymbol{q} =  (q_{N,j})_{j \in \llbracket 1,2^{N} \rrbracket}^\top\in \mathbb{R}^{2^{N} \times 1},
$$
\item the vector $\boldsymbol{\xi}$ representing the resistances of the tree, \textit{i.e.} 
$$
\boldsymbol{\xi} =  (\xi_{1,1},\xi_{1,2}, \ldots,\xi_{N,1},\ldots,\xi_{N,2^{N}})^\top \in (\mathbb{R}_+^*)^{(2^{N+1}-2) \times 1}.
$$
\end{itemize}

\begin{dfntn}
Given two positive integers $i$ and $j$ and their binary expressions
$$
i=\sum_{k=0}^\infty \alpha_k 2^k, \ j=\sum_{k=0}^\infty \beta_k 2^k \textrm{ with } (\alpha_k,\beta_k) \in \{0,1\}^2, \ \forall k\in\mathbb{N} ,
$$
we define $\nu_{i,j}$ as
$$
\nu_{i,j} = \min\{k \in\mathbb{N} \mid \alpha_l = \beta_l, \forall l \geq k\}.
$$
\end{dfntn}
Thanks to Poiseuille's law, we are able to establish a linear relationship between the vectors $\boldsymbol{p}$ and $\boldsymbol{q}$.
\begin{prpstn}\label{poiseuilleRelationProp}
Let $N\in\mathbb{N}^*$. One has
\begin{equation}
\label{poiseuilleRelation}
 p_0 \boldsymbol{u_N} - \boldsymbol{p} = A_N(\boldsymbol{\xi}) \boldsymbol{q},
\end{equation}
where 
\begin{itemize}
\item the symmetric matrix $A_N(\boldsymbol{\xi}) \in \mathbb{R}^{2^{N} \times 2^{N}}$ is called resistance matrix of the tree and is defined by
\begin{equation} 
A_N(\boldsymbol{\xi}) = (a_{i,j}^N)_{1 \leq i,j \leq 2^N}, \ a_{i,j}^N = \left\{
\begin{array}{ll}
\displaystyle r_0+\sum_{(k,l) \in \Pi_{0 \to (N,i)}(N - \nu_{i-1,j-1})} \frac{r_0}{\xi_{k,l}} & \textrm{if }\nu_{i-1,j-1} > 0\\
r_0 & \textrm{if }\nu_{i-1,j-1} = 0\\
\end{array}
\right.
\end{equation}
\item $\boldsymbol{u_N} \in \mathbb{R}^{2^{N} \times 1}$ denotes the ones vector \textit{i.e.}~
$\boldsymbol{u_N} = (1, \ldots, 1)^\top $.
\end{itemize}
\end{prpstn}
\begin{rmrk}
$\Pi_{0 \to (N,i)}(N - \nu_{i-1,j-1})$ is the subpath corresponding to the intersection of the two paths $\Pi_{0 \to (N,i)}$ and $\Pi_{0 \to (N,j)}$.
\end{rmrk}
This proposition is proved in appendix \ref{app1-1}. A close result has already been proved in \cite{mauryMeunier}.

The quantity 
\begin{equation}\label{defNRJ}
\mathcal{E}(\boldsymbol{q},\xi) = \boldsymbol{q}^\top A_N(\boldsymbol{\xi})\boldsymbol{q}
\end{equation} 
corresponds to the total viscous energy dissipated by the fluid in the tree during one second, see \cite{mauryMeunier, mauroy-meunier}, \textit{i.e.}

\begin{equation}
\mathcal{E}(\boldsymbol{q},\boldsymbol{\xi})
                = r_0\left(\Phi^2+\frac{\left(\sum_{j = 1}^{2^{N-1}} q_{N,j}\right)^2}{\xi_{1,1}} + \frac{\left(\sum_{j = 2^{N-1} + 1}^{2^{N}}q_{N,j}\right)^2}{\xi_{1,2}} + \dots + \sum_{j = 1}^{2^N} \frac{q_{N,j}^2}{\xi_{N,j}}+\sum_{j=1}^{2^N} q_{N,j}^2\right).       
\end{equation}

Using the fact that the flow rate is conserved all along the pipe, we rewrite the dissipated energy as
$$
\mathcal{E}(\boldsymbol{q},\boldsymbol{\xi})
                = r_0\Phi^2+\sum_{(i,j)\in \mathcal{B}_N}r_0\frac{q_{i,j}^2}{\xi_{i,j}}.       
$$

As a consequence, $A_N(\boldsymbol{\xi})$ is a symmetric positive definite matrix. It justifies the following proposition. 

\begin{prpstn}\label{Ainvertible}
The resistance matrix $A_N(\boldsymbol{\xi}) \in \mathbb{R}^{2^{N} \times 2^{N}}$ is invertible.
\end{prpstn}

\begin{xmpl}
So as to have an idea of the structure of the matrix $A_N(\boldsymbol{\xi})$, let us use an example of a tree with $N = 2$ levels (thus, with $N+1 = 3$ generations, $2^{N-1} = 4$ outlets and $2^{N} - 1 = 7$ branches). Its resistance matrix $A_2(\boldsymbol{\xi}) \in \mathbb{R}^{4 \times 4} $ is defined as follows:
$$
A_{2}(\boldsymbol{\xi}) = r_0
\left (\begin{array}{cccc}
1+\frac{1}{\xi_{1,1}} + \frac{1}{\xi_{2,1}} & 1+\frac{1}{\xi_{1,1}} & 1 & 1 \\
1+\frac{1}{\xi_{1,1}} & 1+\frac{1}{\xi_{1,1}} + \frac{1}{\xi_{2,2}} & 1 & 1 \\

1 & 1 & 1+\frac{1}{\xi_{1,2}}+\frac{1}{\xi_{2,3}} & 1+\frac{1}{\xi_{1,2}} \\
1 & 1 & 1+\frac{1}{\xi_{1,2}} & 1+\frac{1}{\xi_{1,2}} + \frac{1}{\xi_{2,4}}\\
\end{array} \right)
$$
The tree corresponding to this example is drawn on the figure \ref{treePicture}.\\
\end{xmpl}

\subsection{Boundary conditions}


For biological systems like the vascular network or the lung, it is very difficult to know what are the correct boundary conditions since there exists complex feedback loops that are able to modify the forces applied by muscles relatively to the behaviour of physiological values. Typically, lungs ventilation is controlled not only by carbon dioxide and oxygen concentration in blood but also by blood pH and by mechanical sensors distributed along the tree \cite{Raux}. Concerning boundary conditions in the frame of fluid mechanics, one can refer to \cite{BF} and \cite{Heywood} and in the context of blood flow modeling to \cite{Quarteroni}.

We will investigate the two major types of boundary conditions at the openings of the tree. We will either impose a flow (or parabolic velocity profile) which corresponds to a quantity of fluid that would circulate in the branch, or impose a pressure which corresponds to a force that would be applied to the opening of the branch. If flow rates are imposed at the outlets of the tree, then the pressure will be imposed at the inlet and vice versa so that the problem is well posed. Indeed, the two cases studied in this work are\\ 


\begin{itemize}
\item {\bf 1st case: pressure is imposed at the inlet and flow rates at the outlets.}
\item {\bf 2nd case: flow rate is imposed at the inlet and pressures at the outlets.}\\
\end{itemize}

In the 2nd case, the flows at outlets are not directly accessible. Thus, the following proposition gives the existence of $\boldsymbol{q}$ as the solution of a linear system.

\begin{prpstn}\label{defq}
Assume that the pressures at the outlets of the tree are fixed and the flow in the tree root equal to $\Phi$. 
Then, the vector $\boldsymbol{q} \in \mathbb{R}^{2^N \times 1}$ is the unique solution of the linear system
\begin{equation}
 M_N(\boldsymbol{\xi})\boldsymbol{q} = \boldsymbol{b_N}
\end{equation}
 where 
\begin{equation}
M_N(\boldsymbol{\xi}) = 
\left (\begin{array}{c}
(A_N(\boldsymbol{\xi})\boldsymbol{v_1})^\top \\
\vdots\\
(A_N(\boldsymbol{\xi})\boldsymbol{v_{2^{N}-1}})^\top  \\
\boldsymbol{u_N}^\top 
\end{array} \right) \in \mathbb{R}^{2^{N} \times 2^{N}}, \ 
\boldsymbol{b_N} =
\left (\begin{array}{c}
- \langle \boldsymbol{p},\boldsymbol{v_1}\rangle\\
\vdots\\
- \langle\boldsymbol{p},\boldsymbol{v_{2^{N}-1}}\rangle\\
\Phi\\
\end{array} \right)
\in \mathbb{R}^{2^{N} \times 1}
\end{equation}
and for all $i \in \llbracket 1,2^{N}-1 \rrbracket, 
\ \boldsymbol{v_i}  
= 
(0, \ldots, 0, 1, -1, 0, \ldots,0)^\top  
\in  \mathbb{R}^{2^{N} \times 1}$,
with 1 at the $i$-th position and $-1$ at the $(i+1)$-th position.
\end{prpstn}

See appendix \ref{app1-2} for the proof of this proposition.


\section{The optimization problems}
In this section, we will study the two finite-dimensional constrained optimization problems corresponding to two types of boundary conditions:
\begin{itemize}
\item {1st case: pressure is imposed at the inlet and flows at the outlets,}
\item {2nd case: flow is imposed at the inlet and pressures at the outlets.}\\
\end{itemize}

We consider the same rigid dyadic tree as in section \ref{The model} with $N+1$ generations ($N \in \mathbb{N}, \ N \geq 1$) through which flows the fluid introduced previously. The tree is characterized by its geometry $\boldsymbol{\xi}$ ($\boldsymbol{\xi} \in (\mathbb{R}_+^*)^{(2^{N+1}-2) \times 1})$, its resistance matrix $A_N(\boldsymbol{\xi})$ ($A_N(\boldsymbol{\xi}) \in \mathbb{R}^{2^{N} \times 2^{N}}$) and the volumetric flow rates $\boldsymbol{q}$ ($\boldsymbol{q} \in \mathbb{R}^{2^N \times 1}$) at its outlet. We recall that, at outlets, pressures are linked to flow rates with the relation $\boldsymbol{p}=A_N(\boldsymbol{\xi}) \boldsymbol{q}$.
\par We want to minimize the total viscous dissipated energy $\mathcal{E}$ defined by (\ref{defNRJ}) with respect to the pair $(\boldsymbol{q},\boldsymbol{\xi})$ under the constraints
\begin{itemize}
\item that $\boldsymbol{q}$ is the flow vector at outlets of the tree, when the Poiseuille's model is considered. This condition can always be treated as a constraint, it takes however different forms depending on the case considered.
\item and that the total volume $V$ of the tree is constrained, which from (\ref{defVol}), writes
\begin{equation}\label{cond1}
1+\sum_{(i,j) \in \mathcal{B}_N} \xi_{i,j} = \Lambda 
\end{equation}
with $\Lambda = \frac{V}{\pi R_0^2 L_0} > 1$.
\end{itemize} 


To make the admissible set of pairs $(\boldsymbol{q},\boldsymbol{\xi})$ clear, we define the intermediate set $\mathscr{A}_\Lambda$ as
\begin{equation}\label{ALambda}
\mathscr{A}_\Lambda = \left\{ 
\boldsymbol{\xi} \in (\mathbb{R}_+^*)^{(2^{N+1}-2) \times 1} \mid 
\boldsymbol{\xi} \textrm{ verifies the conditions }   (\ref{cond2})\textrm{ and }(\ref{cond1})
\right \}.
\end{equation}

\subsection{1st case: pressures at the inlet and flow rates at the outlets are known}\label{vconnu}

The first case is much simpler than the second one (presented in section \ref{pconnu}). The pressure is imposed at the inlet along with the flows vector $\boldsymbol{q}$ at the outlets, thus the optimization problem is
\begin{equation}
(\mathscr{P}_1) \ \left\{
\begin{array}{l}
\min \mathcal{E}(\boldsymbol{q},\boldsymbol{\xi})\\
\boldsymbol{\xi} \in \mathscr{A}_\Lambda ,
\end{array}\right.
\end{equation}
where $\boldsymbol{q}\in \mathbb{R}_+^{2^N\times 1}$ is assumed fixed.
\par  In this case, since the volumetric flow rate is conserved along the tree, every value of the intermediate flow rate $q_{i,j}$ is known from the values of $q_{N,j}$, $j\in \llbracket 1,2^N\rrbracket$.

\par The existence of solutions for the problem $(\mathscr{P}_1) $ is classical. Indeed, if one of the optimization variables goes to zero then the energy tends to a positive infinite value. It is thus possible to come down to minimize $\mathcal{E}(\boldsymbol{q},\cdot)$ on a compact set, which yields the existence of a solution. Moreover, $\mathscr{A}_\Lambda$ is a convex set and $\mathcal{E}(\boldsymbol{q},\cdot)$ is strictly convex, this ensures the uniqueness of the minimizer for the problem $(\mathscr{P}_1) $.
\par We will now write the first order optimality conditions: we denote by $\boldsymbol{\xi^*}$ the minimizer of the problem $(\mathscr{P}_1)$. There exists a Lagrange multiplier $\lambda\in\mathbb{R}$ such that
$$
\forall (i,j)\in \mathcal{B}_N, \ -\frac{q_{i,j}^2}{{\xi_{i,j}^*}^2}=\lambda.
$$
Using the volume constraint, we immediately obtain the following expression for the minimizer $\boldsymbol{\xi^*}$:
$$
\forall (i,j)\in \mathcal{B}_N, \ \xi_{i,j}^*=(\Lambda-1)\frac{|q_{i,j}|}{\sum_{(i,j)\in \mathcal{B}_N}|q_{i,j}|}.
$$
The following proposition summarizes the conclusions for the first case.
\begin{prpstn}
The problem $(\mathscr{P}_1)$ has a unique solution $\boldsymbol{\xi^*}$ given by
$$
\forall (i,j)\in \mathcal{B}_N, \ \xi_{i,j}^*=(\Lambda-1)\frac{|q_{i,j}|}{\sum_{(i,j)\in \mathcal{B}_N}|q_{i,j}|}.
$$
\end{prpstn}

\subsection{2nd Case: flow rate at the inlet and the pressures at the outlets are known}\label{pconnu}

\par The volumetric flow rate $\Phi>0$ which enters the overall tree through the root branch and the pressures at the outlets of the tree (\textit{i.e.} ~ the vector $\boldsymbol{p}$, $\boldsymbol{p} \in (\mathbb{R}_+^*)^{2^N \times 1}$) are assumed to be fixed.
\par According to Proposition \ref{defq}, the constraint $\boldsymbol{q}=\left(M_N(\boldsymbol{\xi})\right)^{-1}\boldsymbol{b_N}$ comes down to impose an incompressible and Poiseuille-like flow through the tree.
\par Now, let us define the set $\mathscr{U}_\Lambda$ of admissible pairs $(\boldsymbol{q}, \boldsymbol{\xi})$ as
\begin{equation}
\mathscr{U}_\Lambda =
\left\{
(\boldsymbol{q},\boldsymbol{\xi}) \mid \boldsymbol{\xi} \in \mathscr{A}_\Lambda \textrm{ and } \boldsymbol{q}=\left(M_N(\boldsymbol{\xi})\right)^{-1}\boldsymbol{b_N}
\right \}.
\end{equation}
The resultant optimization problem writes 
\begin{equation}
(\mathscr{P}_2) \left \{
\begin{array}{ll}
\min \mathcal{E}(\boldsymbol{q}, \boldsymbol{\xi})\\
(\boldsymbol{q},\boldsymbol{\xi}) \in \mathscr{U}_\Lambda.
\end{array}
\right.
\end{equation}
We claim that generically with respect to the vector $\boldsymbol{p}$, the problem $(\mathscr{P}_2)$ has no solution. Nevertheless, in the degenerate situation where all the pressures at the outlet of the tree are equal, the problem $(\mathscr{P}_2)$ has a unique solution.

\par We will start with the case where $(\mathscr{P}_2)$ has a solution. The following theorem gives a characterization of this situation.
\begin{thrm}\label{theoExistenceSol}
The problem $(\mathscr{P}_2)$ has a solution if, and only if there exists $\pi^*\in \mathbb{R}$ such that
$$
\boldsymbol{p}= \pi^*\boldsymbol{u_N},
$$

where $\boldsymbol{u_N}\in \mathbb{R}^{2^N\times 1}$ denotes the unit vector $(1,...,1)^T$. Furthermore,
$$
\min_{(\boldsymbol{q},\boldsymbol{\xi}) \in \mathscr{U}_\Lambda} \mathcal{E}(\boldsymbol{q}, \boldsymbol{\xi})=r_0 \Phi^2 \left(1+ \frac{N^2}{\Lambda - 1} \right).
$$
\end{thrm}

This theorem emphasizes the fact that if two pressures at the outlets are different, then the problem $(\mathscr{P}_2)$ has no solution. In this case, we are able to exhibit a minimizing sequence. It has to be noticed that the value of the infimum does not differ according to the case considered.

\begin{thrm}\label{theorem}
Assume that at least two pressures differ at the outlets of the tree, \textit{i.e.}
\begin{equation}\label{assumpPressure}
\exists j \in \llbracket 1 , 2^{N} \rrbracket \mid p_{N,j} \neq p_{N,j+1}.
\end{equation}
Then,
\begin{enumerate}
\item the problem $(\mathscr{P}_2)$ has no solution,
\item a minimizing sequence $(\boldsymbol{\xi}^\varepsilon,\boldsymbol{q}^\varepsilon)_{\varepsilon > 0}$ is given by 
\begin{equation}\label{minSeq}
\left \{
\begin{array}{ll}
\xi_{i,j}^\varepsilon =  \frac{\Lambda - 1}{N} - \left( \frac {2^{N+1} - 2}{N} - 1 \right) \varepsilon \ & \forall (i,j) \in \Pi_{0 \to (N,1)},\\
\xi_{i,j}^\varepsilon =  \varepsilon &  \forall (i,j) \in  \mathcal{B}_N \backslash \Pi_{0 \to (N,1)},\\
\boldsymbol{q}^\varepsilon = M_N(\boldsymbol{\xi}^\varepsilon)^{-1} \boldsymbol{b_N},
\end{array}
\right.
\end{equation}
with $\varepsilon>0$ sufficiently small to respect the constraint: $ \forall (i,j) \in \mathcal{B}_N, \ \xi_{i,j}^{\varepsilon} > 0$,
\item one has
$$
\inf _{(\boldsymbol{q},\boldsymbol{\xi}) \in \mathscr{U}_\Lambda}\mathcal{E}(\boldsymbol{q}, \boldsymbol{\xi}) =r_0 \Phi^2 \left(1+ \frac{N^2}{\Lambda - 1} \right).
$$
\item the sequence $(\boldsymbol{q}^\varepsilon)_{\varepsilon > 0}$ converges to the vector $(\Phi,0, \dots , 0)^\top$ as $\varepsilon\to 0$.
\end{enumerate}
\end{thrm}

The limit case 
``$\boldsymbol{q^0}=(\Phi,0, \dots , 0)^\top$'' corresponds to the case where the fluid exits the tree only by the outlet denoted by $(N,1)$. The use of some symmetry properties in the structure of the objective function $\mathcal{E}$ yields that we would easily get another minimizer by choosing any other outlet as main exit of the fluid, which would correspond to the closure of every path linking the root branch to the outlet except $\Pi_{0\to (N,i)}$, for any $i\in \llbracket 1,2^{N}\rrbracket$.
Note that the symmetry property does not hold for the pressures at outlets. However, when we take the limit in the minimizing sequence, since the total flow in the tree is imposed, the pressure drop between the root and the outlet that remains open is the same whatever the outlet chosen. Thus, the pressure at root will depend on the pressure imposed at the outlet. Hence, introducing the matrix $M_N(\boldsymbol{\xi})$ that does not depend on the root pressure is well adapted.

\par The optimal shape for the case with each pressures equal at outlets is obviously unstable relatively to the pressures: a slight change in one of them will shift the optimal shape from a tree-like structure to a pipe-like structure. From the modeling point of view, it means that a tree-like structure verifying the hypothesis of this section should not be encountered in nature as a result of an evolution process such as natural selection. Indeed a perfect adjustment of pressure at outlets is very difficult to assure and maintain. 

The proofs of the theorems will be decomposed into two steps, whose details can be found in the next section:
\begin{enumerate}
\item determination of a lower bound for the total viscous dissipated energy,
\item construction of the minimizing sequence $(\boldsymbol{q}^\varepsilon,\boldsymbol{\xi}^\varepsilon)_{\varepsilon > 0}$.
\end{enumerate}

\section{Proofs of theorems \ref{theoExistenceSol} and \ref{theorem} (2nd case)}

\subsection{Determination of a lower bound for the total viscous dissipated energy}\label{SubSectionLowerBound}

We will now focus on an auxiliary optimization problem whose resolution is useful in the proof of Theorem \ref{theorem}. Let $\boldsymbol{q} \in \mathbb{R}^{(2^{N+1}-2) \times 1}$ and $\boldsymbol{\xi} \in (\mathbb{R}_+^*)^{(2^{N+1}-2) \times 1}$ be the two vectors
$$
\boldsymbol{q} = (q_{1,1},q_{1,2}, \dots, q_{N,1}, \dots, q_{N,2^N})^\top \textrm{ and }
\boldsymbol{\xi} = (\xi_{1,1},\xi_{1,2}, \dots, \xi_{N,1}, \dots, \xi_{N,2^N})^\top.
$$
Notice that, in the rest of the paper and besides Section \ref{SubSectionLowerBound}, the notation $\boldsymbol{q}$ points out the vector composed only of the flow rates at the outlets of the tree.
\par Let $r_0>0$, $\Phi>0$ and $\Lambda>1$. Let $\mathscr{(Q)}$ be the finite-dimensional constrained optimization problem

\begin{equation}
\mathscr{(Q)} \left \{
\begin{array}{ll}
\min \mathcal{E}(\boldsymbol{q}, \boldsymbol{\xi})\\
(\boldsymbol{q},\boldsymbol{\xi})\in \mathscr{C}_1\times \mathscr{C}_2,
\end{array}
\right.
\end{equation}

where $\mathcal{E}(\boldsymbol{q}, \boldsymbol{\xi})$ is defined by
\begin{equation}
\mathcal{E}(\boldsymbol{q}, \boldsymbol{\xi}) = r_0\Phi^2+\sum_{(i,j) \in \mathcal{B}_N}  r_0 \frac{q_{i,j}^2}{\xi_{i,j}},
\end{equation}
and the sets of constraints are
\begin{eqnarray*}
 \mathscr{C}_1 & = & \left\{\boldsymbol{q}  \in (\mathbb{R})^{(2^{N+1}-2) \times 1} \mid \forall i \in \llbracket 1, N \rrbracket, \
\displaystyle \sum_{j=1}^{2^i} q_{i,j} = \Phi\right\}, \\
 \mathscr{C}_2 & = & \left\{\boldsymbol{\xi} \in (\mathbb{R}_+^*)^{(2^{N+1}-2) \times 1} \mid \displaystyle\sum_{(i,j) \in \mathcal{B}_N} \xi_{i,j}= \Lambda -1\right\}.
\end{eqnarray*}

\begin{prpstn}\label{pbAuxQ}
The problem $\mathscr{(Q)}$ has a solution $(\boldsymbol{q}^*, \boldsymbol{\xi}^*)$ verifying necessarily
\begin{equation}\label{optCondQ}
\forall (i,j) \in \mathcal{B}_N, \
\frac{q_{i,j}^*}{\xi_{i,j}^*} = \frac{N \Phi}{\Lambda - 1} .
\end{equation}
Moreover, the value of the minimum is equal to $ r_0 \Phi^2 \left( 1+ \frac{N^2}{\Lambda -1}\right)$.
\end{prpstn}

Note that there is no reason to have uniqueness of the solution of Problem $(\mathscr{Q})$.

\begin{proof}
The positivity of $\mathcal{E}(\cdot,\cdot)$ 
grants
the existence of the lower bound $\displaystyle \inf \{\mathcal{E}(\boldsymbol{q}, \boldsymbol{\xi}),(\boldsymbol{q}, \boldsymbol{\xi}) \in \mathscr{C}_1\times \mathscr{C}_2 \}$. 
\par Since the constraints are uncoupled, one has
\begin{equation}
\inf_{(\boldsymbol{q},\boldsymbol{\xi}) \in \mathscr{C}_1\times \mathscr{C}_2 } \mathcal{E}(\boldsymbol{q}, \boldsymbol{\xi}) =
\inf_{\boldsymbol{\xi} \in \mathscr{C}_2} \inf_{\substack{\boldsymbol{q} \in \mathscr{C}_1}} \mathcal{E}(\boldsymbol{q}, \boldsymbol{\xi}).
\end{equation}

$\boldsymbol{\xi}\in\mathscr{C}_2$ being fixed, let us first focus on the optimization problem
\begin{equation}
(\mathscr{Q}_{\boldsymbol{\xi}}) \left \{
\begin{array}{ll}
\min \mathcal{E}(\boldsymbol{q}, \boldsymbol{\xi})\\
\boldsymbol{q}  \in \mathscr{C}_1
\end{array}
\right.
\end{equation}
Since the energy $\mathcal{E}(\cdot ,\boldsymbol{\xi})$ is obviously continuous, coercive and strictly convex, on the convex closed set $\mathscr{C}_1$, the problem $(\mathscr{Q}_{\boldsymbol{\xi}})$ has a unique solution $\boldsymbol{q}(\boldsymbol{\xi})$.

By virtue of Kuhn-Tucker's theorem, there exists a real Lagrange multiplier $\mu_1$ such that
$$
\frac{\partial \mathcal{E}}{\partial q_{1,1}} = \frac{2 r_0 q_{1,1}(\boldsymbol{\xi})}{\xi_{1,1}} = \mu_1 \textrm{ and }
\frac{\partial \mathcal{E}}{\partial q_{1,2}} = \frac{2 r_0 q_{1,2}(\boldsymbol{\xi})}{\xi_{1,2}} = \mu_1.
$$
Since $q_{1,1}(\boldsymbol{\xi}) + q_{1,2}(\boldsymbol{\xi}) = \Phi$, one has
$$
q_{1,1}(\boldsymbol{\xi}) = \frac{\xi_{1,1}}{\xi_{1,1} + \xi_{1,2}} \Phi \textrm{ and }
q_{1,2}(\boldsymbol{\xi}) = \frac{\xi_{1,2}}{\xi_{1,1} + \xi_{1,2}} \Phi.
$$

Similarly, 
\begin{equation}\label{secoDeb}
\forall i \in \llbracket 1,N\rrbracket, \
\forall j \in \llbracket 1,2^i\rrbracket, \
q_{i,j}(\boldsymbol{\xi}) = \frac{ \xi_{i,j}}{\sum_{k=1}^{2^i} \xi_{i,k}} \Phi.
\end{equation}
Let us now focus on the minimization of the function $\boldsymbol{\xi}\in\mathscr{C}_2\mapsto \mathcal{E}(\boldsymbol{q}(\boldsymbol{\xi}),\boldsymbol{\xi})$. From (\ref{secoDeb}), one gets
\begin{eqnarray*}
\mathcal{E}(\boldsymbol{q}(\boldsymbol{\xi}),\boldsymbol{\xi})& = & r_0\Phi^2 \left(1+\sum_{(i,j)\in \mathcal{B}_N}\frac{\xi_{i,j}}{\left(\sum_{k=1}^{2^i} \xi_{i,k}\right)^2}\right)\\
& = & r_0 \Phi^2 
\left(1 +\frac{1}{ \xi_{1,1} + \xi_{1,2}} + \dots +\frac{1}{\sum_{i=1}^{2^N} \xi_{N,i}} 
\right).
\end{eqnarray*}
\par Let us introduce the change of variables
\begin{equation}
y_i = \displaystyle \sum_{j=1}^{2^i} \xi_{i,j}, \ i \in \llbracket 1,2^N \rrbracket.
\end{equation}
Then, our optimization problem becomes
\begin{equation}
\left \{
\begin{array}{ll}
\min \left[r_0 \Phi^2 \left(1 + \displaystyle\sum_{i=1}^{N} \frac{1}{y_i} \right) \right]\\
\displaystyle\sum_{i=1}^{N} y_i = \Lambda - 1.
\end{array}
\right.
\end{equation}

A direct application of Kuhn-Tucker's theorem yields that the unique minimizer of this problem is $\boldsymbol{y^*} = (\frac{\Lambda - 1}{N}, \dots ,\frac{\Lambda - 1}{N})$ and that the minimum is equal to $r_0 \Phi^2 \left(1 + \frac{N^2}{\Lambda -1} \right)$.

From Equation (\ref{secoDeb}) we know that $q^*_{i,j} = \frac{\xi^*_{i,j}}{y_i^*} \Phi$, and we deduce equation (\ref{optCondQ}): $\frac{q^*_{i,j}}{\xi_{i,j}^*} = \frac{N \Phi}{\Lambda-1}$.

\par Consequently, the minimum of $\mathcal{E}$ in $\mathscr{C}_1\times \mathscr{C}_2$ is $ r_0 \Phi^2 \left( 1+ \frac{N^2}{\Lambda -1}\right)$ and the conclusion of this proposition follows.
 \end{proof}
\begin{rmrk}
The expression of the necessarily first-order optimality conditions in the proof of Proposition \ref{pbAuxQ} leads easily to conclude that Problem $(\mathscr{Q})$ has an infinite number of minimizers.
\par Moreover, an interesting minimizing sequence $(\boldsymbol{q}^\varepsilon,\boldsymbol{\xi}^\varepsilon)_{\varepsilon >0 }$ in $\mathscr{C}_1\times \mathscr{C}_2$, of $\mathcal{E}$ for our problem is given by
\begin{equation}
\forall \varepsilon >0, \ \left \{
\begin{array}{ll}
\forall (i,j) \in \Pi_{0 \to (N,1)}, & q_{i,j}^\varepsilon = \Phi \textnormal{ and }\xi_{i,j}^\varepsilon =  \frac{\Lambda - 1}{N} - \left( \frac {2^{N+1} - 2}{N} - 1 \right) \varepsilon\\
\forall (i,j) \in  \mathcal{B}_N \backslash \Pi_{0 \to (N,1)}, & q_{i,j}^\varepsilon = 0 
\textnormal{ and }
\xi_{i,j}^\varepsilon =  \varepsilon. \\
\end{array}
\right. 
\end{equation}
Physically, it corresponds to the case in which the fluid exits the tree only by the outlet denoted by $(N,1)$.
\end{rmrk}

\subsection{Proof of Theorem \ref{theoExistenceSol}}

As a preliminary to the proof, let us deal with the qualification issue for Problem $(\mathscr{P}_2)$. In fact, any elements of the set $\mathscr{U}_\Lambda$ verifies the constraint qualifications. Indeed, thanks to Proposition \ref{defq}, one can write $\boldsymbol{q}=M_N(\boldsymbol{\xi})^{-1}\boldsymbol{b_N}$ and furthermore, it is obvious that any element of the set $\mathscr{A}_\Lambda$ verifies the constraint qualifications. Combining the two previous remarks, the set $\mathscr{U}_\Lambda$ inherits from this property .

\par Let us assume that Problem $(\mathscr{P}_2)$ has a solution $(\boldsymbol{q^*},\boldsymbol{\xi^*})$. In the proof of Theorem \ref{theorem} presented in Section \ref{proofTheorem}, it is proved without any additional assumption than those of Theorem \ref{theoExistenceSol}, that the sequence $(\boldsymbol{q}^\varepsilon,\boldsymbol{\xi}^\varepsilon)_{\varepsilon >0}$ defined by (\ref{minSeq}) is a minimizing sequence for Problem $(\mathscr{P}_2)$, whatever values of pressures at the outlets. Hence it follows that
$$
\mathcal{E}(\boldsymbol{q^*},\boldsymbol{\xi^*})=\lim_{\varepsilon \to 0}\mathcal{E}(\boldsymbol{q}^\varepsilon,\boldsymbol{\xi}^\varepsilon)=r_0 \Phi^2 \left(1+ \frac{N^2}{\Lambda - 1} \right).
$$
Now, recall that, by Proposition \ref{pbAuxQ}, any solution of Problem $\mathscr{(Q)}$ realizes the minimum $r_0 \Phi^2 \left( 1 + \frac{N^2}{\Lambda - 1}\right)$ and verifies the necessary optimality conditions (\ref{optCondQ}). Since $\mathcal{E}(\boldsymbol{q^*},\boldsymbol{\xi^*})=r_0 \Phi^2 \left(1+ \frac{N^2}{\Lambda - 1} \right)$ and $(\boldsymbol{q^*},\boldsymbol{\xi^*})\in \mathscr{U}_\Lambda \subset \mathscr{C}_1\times \mathscr{C}_2$, the pair $(\boldsymbol{q^*},\boldsymbol{\xi^*})$ belongs to the set of minimizers of Problem $\mathscr{(Q)}$ and verifies therefore, the necessary optimality conditions (\ref{optCondQ}), namely
\begin{equation}
\forall (i,j) \in \mathcal{B}_N, \
q_{i,j}^* = \frac{N \Phi}{\Lambda - 1} \xi_{i,j}^*.
\end{equation}

Let us denote as previously by $p_0$ the pressure at the inlet of the tree. According to Poiseuille's law and reasoning by induction, we obtain
$$
\left\{\begin{array}{l}
\left.\begin{array}{r}
p_0-p_{1,1}=\frac{N\Phi}{\Lambda-1}\\
p_{0}-p_{1,2}=\frac{N\Phi}{\Lambda-1}
\end{array}\right\}\Longrightarrow p_{1,1}=p_{1,2}\\
~\hspace{1cm}\vdots \\
\left.\begin{array}{r}
p_{i,j}-p_{i+1,2j-1}=\frac{N\Phi}{\Lambda-1}\\
p_{i,j}-p_{i+1,2j}=\frac{N\Phi}{\Lambda-1}
\end{array}\right\}\Longrightarrow p_{i+1,2j-1}=p_{i+1,2j}, \forall (i,j)\in \displaystyle \bigcup_{i\in\llbracket 1,N-1\rrbracket}\mathcal{B}_{N,i}.
\end{array}
\right.
$$
In particular, it implies that
$$
p_{N,1}=p_{N,2}=\dots=p_{N,2^N-1}=p_{N,2^N}.
$$
Conversely, let us prove that, if there exists $\pi^*\in\mathbb{R}$ such that $\boldsymbol{p}=\pi^*\boldsymbol{u_N}$, then, Problem $(\mathscr{P}_2)$ has a solution. In this case, we are able to exhibit an element $(\boldsymbol{q^*},\boldsymbol{\xi^*})$ which belongs to the admissible set $\mathscr{U}_\Lambda$. For instance, let us choose
$$
\forall (i,j)\in \mathcal{B}_N, \ \xi_{i,j}^*=\frac{\Lambda-1}{N2^i}.
$$
Then, it is easy to make the calculation of the flow throughout the tree. Indeed
$$
\left.\begin{array}{r}
p_{N-1,1}-p_{N,1}=\frac{q_{N,1}^*}{\xi_{N,1}^*}\\
p_{N-1,1}-p_{N,2}=\frac{q_{N,2}^*}{\xi_{N,2}^*}
\end{array}\right\}\Longrightarrow q_{N,1}^*=q_{N,2}^*\textnormal{ since }p_{N,1}=p_{N,2}=\pi^*\textnormal{ and }\xi_{N,1}^*=\xi_{N,2}^*.
$$
Iterating this reasoning proves that $\boldsymbol{q^*}$ is proportional to $\boldsymbol{u_N}$ and finally, by induction, we obtain
$$
\forall (i,j)\in\mathcal{B}_N, \ q_{i,j}^*=\frac{\Phi}{2^i}.
$$
A direct calculation shows hence that
$$
\mathcal{E}(\boldsymbol{q^*},\boldsymbol{\xi^*})=r_0 \Phi^2 \left(1+ \frac{N^2}{\Lambda - 1} \right).
$$

\subsection{Proof of Theorem \ref{theorem}}\label{proofTheorem}
The fact that Problem $(\mathscr{P}_2)$ has no solution when assumption (\ref{assumpPressure}) holds is a direct consequence of Theorem \ref{theoExistenceSol}. In this section, we prove that the sequence exhibited in (\ref{minSeq}) is a minimizing sequence for Problem $(\mathscr{P}_2)$, in other words,
\begin{equation}\label{ConvMinSeq}
\lim_{\varepsilon \to 0} \mathcal{E}(\boldsymbol{q}(\boldsymbol{\xi}^\varepsilon),\boldsymbol{\xi}^\varepsilon) = r_0 \Phi^2 \left(1+ \frac{N^2}{\Lambda - 1} \right), 
\end{equation}
where $\boldsymbol{q}(\boldsymbol{\xi}^\varepsilon)$ is the unique solution of the linear system $ M_N(\boldsymbol{\xi}^\varepsilon)\boldsymbol{q} = \boldsymbol{b_N}$ (see Proposition \ref{defq}). Indeed, since $\mathscr{U}_\Lambda \subset \mathscr{C}_1\times \mathscr{C}_2$ and since the sequence $(\boldsymbol{q}(\boldsymbol{\xi}^\varepsilon),\boldsymbol{\xi}^\varepsilon)_{\varepsilon >0}$ has been constructed so that its elements belong to $\mathscr{U}_\Lambda$, one has for $\varepsilon$ sufficiently small,
$$
\mathcal{E}(\boldsymbol{q}(\boldsymbol{\xi}^\varepsilon),\boldsymbol{\xi}^\varepsilon) \geq r_0 \Phi^2 \left(1+ \frac{N^2}{\Lambda - 1} \right).
$$
Hence, proving (\ref{ConvMinSeq}) will yield at the same time that $(\boldsymbol{q}(\boldsymbol{\xi}^\varepsilon),\boldsymbol{\xi}^\varepsilon)_{\varepsilon >0}$ is a minimizing sequence for Problem $(\mathscr{P}_2)$ and that the infimum is equal to $r_0 \Phi^2 \left(1+ \frac{N^2}{\Lambda - 1} \right)$.\\

Let us denote for the sake of clarity, $\boldsymbol{q}(\boldsymbol{\xi}^\varepsilon)$ by $\boldsymbol{q}^\varepsilon$, $A_N(\boldsymbol{\xi}^\varepsilon)$ by $A_N(\varepsilon)$ and $M_N(\boldsymbol{\xi}^\varepsilon)$ by $M_N(\varepsilon)$. Now, recall that, according to Proposition \ref{defq}, the vector $\boldsymbol{q}^\varepsilon$ is completely characterized by the system

\begin{equation}
\left\{
\begin{array}{l}
\langle \boldsymbol{q}^\varepsilon, A_N(\varepsilon) \boldsymbol{v_i}\rangle = - \langle \boldsymbol{p},\boldsymbol{v_i}\rangle ,~\forall i \in \llbracket 1,2^N-1 \rrbracket \\
\langle \boldsymbol{q}^\varepsilon,\boldsymbol{u_N} \rangle = \Phi .
\end{array}
\right.
\end{equation}
Let  $\varepsilon>0$. The matrix $A_N(\varepsilon)$ admits the decomposition
\begin{align}
 A_{N}(\varepsilon) = r_0 \left(U_N + \frac{1}{\varepsilon} \widetilde{A}_N^1 + \frac{1}{\frac{\Lambda-1}{N} - \left(\frac{2^{N+1}-2}{N}-1\right) \varepsilon} \widetilde{A}_N^2 \right)
\end{align}
where $U_N$, $\widetilde{A}_N^1$ and $\widetilde{A}_N^2$ are three matrices independent of $\varepsilon$ with same size as $A_N(\varepsilon)$, such that
\begin{itemize}
\item The first part, $U_N$ is the matrix of size $2^N \times 2^N$ whose coefficients are all equal to $1$. This term comes from the root branch that add a resistance $r_0$ to the resistance of any pathway in the tree.
\item The second part, $\frac{r_0}{\varepsilon} \widetilde{A}_N^1$, corresponds to pathways consisting only in branches $(i,j)$ such that $\xi_{(i,j)}=\varepsilon$ (i.e. whose diameter will go to $0$ with $\varepsilon$), namely all pathways in the tree except those included in the pathway going from the root to the outlet $(N,1)$. $\widetilde{A}_N^1$ is more precisely characterized in Lemma \ref{decompAN} below.
\item The third part $\frac{1}{\frac{\Lambda-1}{N} - (2^N-N) \varepsilon} \widetilde{A}_N^2$ corresponds to the pathways that stay open when $\varepsilon$ goes to $0$.
\end{itemize}

Let us define $\widetilde{A}_N(\varepsilon) =  \varepsilon A_N(\varepsilon)$. Then, we have
\begin{equation}
\lim_{\varepsilon \to 0} \widetilde{A}_N(\varepsilon) = r_0 \widetilde{A}_{N}^1.
\end{equation}
Therefore, the vector $\boldsymbol{q}^\varepsilon$ is completely characterized by the system
\begin{equation}
\label{syst2}
\left\{
\begin{array}{l}
\langle \boldsymbol{q}^\varepsilon, \widetilde{A}_N(\varepsilon) \boldsymbol{v_i}\rangle = - \varepsilon \langle \boldsymbol{p},\boldsymbol{v_i}\rangle ,~\forall i \in \llbracket 1,3 \rrbracket \\
\langle \boldsymbol{q}^\varepsilon,\boldsymbol{u_N} \rangle = \Phi
\end{array}
\right.
\end{equation}
Thus it is necessary to study more closely the matrix $\widetilde{A}_{N}^1$ that will give the behaviour of $\boldsymbol{q}^\varepsilon$ when $\varepsilon$ goes to $0$.
\begin{lmm}
\label{decompAN}
$\widetilde{A}_N^1 = (0) \oplus \mathop{\bigoplus}_{p=0}^{N-2} B_p$ where $B_p$ is a resistance matrix of a tree of $p+1$ generations whose branches have all the same resistance equal to $1$.
\end{lmm}
\begin{proof}
From its definition, $\widetilde{A}_N^1$ corresponds to the resistance matrix of a tree with $N+1$ generations where the resistances of the branches on the path from root to outlet $(N,1)$ are all $0$ and the resistance of the other branches are all $1$.
\end{proof}
\begin{psfrags}
\psfrag{_A}{$\widetilde{A}_2^1 =$}
\psfrag{_0}{$(0)$}
\psfrag{_+}{$\oplus$}
\psfrag{_1}{$(1)$}
\psfrag{_2}{$\left(\begin{array}{cc} 2 & 1\\ 1 & 2 \end{array}\right)$}
\psfrag{_3}{$\widetilde{A}_2^1 = \left(
\begin{array}{cccc}
0 & 0 & 0 & 0\\
0 & 1 & 0 & 0\\
0 & 0 & 2 & 1\\
0 & 0 & 1 & 2
\end{array}
\right)$}
\begin{figure}[h!]
\begin{center}
\hspace{-2cm}\includegraphics[height=4cm]{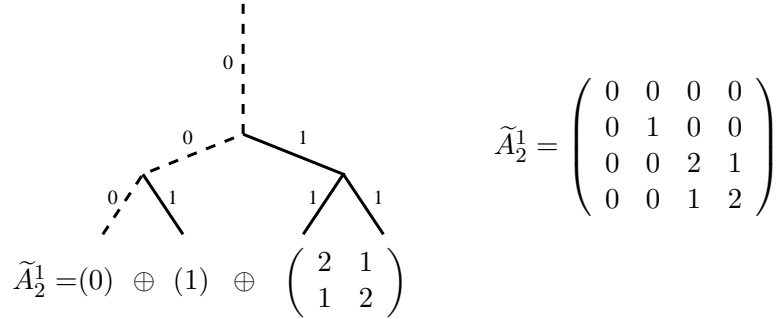}
\end{center}
\hspace{3cm}
\caption{Example of the matrix $\widetilde{A}_N^1$ for $N=2$. The numbers next to the branches represent their resistance. The matrix $\widetilde{A}_N^1$ is the direct sum of the resistance matrices of the subtrees built up by branches with non zero resistance.}
\end{figure}
\end{psfrags}
Now, system (\ref{syst2}) is equivalent to

\begin{equation}\label{M_2system}
\widetilde{M}_N(\varepsilon)\boldsymbol{q}^\varepsilon = \boldsymbol{\widetilde{b}_N}(\varepsilon)
\end{equation}
where
$$
\widetilde{M}_N(\varepsilon) = 
\left (\begin{array}{c}
(\widetilde{A}_N(\varepsilon) \boldsymbol{v_1})^\top \\
\vdots\\
(\widetilde{A}_N(\varepsilon) \boldsymbol{v_{2^N-1}})^\top  \\
\boldsymbol{u_N}^\top
\end{array} \right)
\text{ and }
\boldsymbol{\widetilde{b}_N}(\varepsilon) = 
\left (\begin{array}{c}
- \varepsilon \langle \boldsymbol{p},\boldsymbol{v_1}\rangle\\
\vdots\\
- \varepsilon \langle \boldsymbol{p},\boldsymbol{v_{2^N-1}}\rangle\\
\Phi
\end{array} \right)
$$
One has
$$
\lim_{\varepsilon \to 0} \widetilde{M}_{N}(\varepsilon) = 
\widetilde{M}_{N}
\text{ and }
\lim_{\varepsilon \to 0} \boldsymbol{\widetilde{b}_{N}}(\varepsilon) = 
\boldsymbol{\widetilde{b}_{N}}
$$
with
$$
\begin{array}{ccc}
\widetilde{M}_{N} =
\left (\begin{array}{c}
r_0 (\widetilde{A}_N^1 \boldsymbol{v_1})^\top \\
\vdots\\
r_0 (\widetilde{A}_N^1 \boldsymbol{v_{2^N-1}})^\top  \\
\boldsymbol{u_N}^\top
\end{array} \right)
&
\text{ and }
&
\boldsymbol{\widetilde{b}_N} =
\left (\begin{array}{c}
0 \\
\vdots \\
0 \\
\Phi \\
\end{array} \right)
\end{array}
$$

\begin{lmm}\label{MnInvert}
The matrix $\widetilde{M}_{N}$ is invertible.
\end{lmm}
\begin{proof}
Let us show that the family $\left((\widetilde{A}_N^1 \boldsymbol{v_i})_{i=1 \dots 2^N-1}, \boldsymbol{u_N}\right)$ is linearly independent.
We consider some real numbers $(\lambda_i)_{i=1 \dots 2^N-1}$ and $\mu$ such that
\begin{equation}
\label{libre}
\mu \boldsymbol{u_N}+\sum_{i=1}^{2^{N}-1} \lambda_i \widetilde{A}_N^1 \boldsymbol{v_i} = 0
\end{equation}

Let us recall that, from Lemma \ref{decompAN}, $\widetilde{A}_N^1 = (0)\displaystyle \oplus \mathop{\bigoplus}_{p=0}^{N-2} B_p$ where $B_p$ is  a resistance matrix and is thus invertible.
Consequently, the range of $\widetilde{A}_N^1$ is $\textnormal{Span}\left( \boldsymbol{e_2}, \dots, \boldsymbol{e_{2^N}} \right)$, which implies that, in equation (\ref{libre}) the component along $\boldsymbol{e_1}$ is reduced to $\mu = 0$.

Now, the restriction of $\widetilde{A}_N^1$ on its range is a one-to-one correspondence as soon as the projection of the family $(\boldsymbol{v_i})_{i=1 \dots 2^{N}-1}$ on the range of $\widetilde{A}_N^1$ is a basis of the range of $\widetilde{A}_N^1$, which is clearly true in our case. Then, the family $(\widetilde{A}_N^1 \boldsymbol{v_i})_{i=1 \dots 2^N-1}$ is linearly independent and all $\lambda_i$'s are $0$.
\end{proof}\

Since the map $A \mapsto A^{-1}$ defined on the set of invertible matrices is continuous, the unique solution of $\widetilde{M}_N(\varepsilon) \boldsymbol{q}^\varepsilon = \boldsymbol{\widetilde{b}_N}$ converges to the unique solution of the system $\widetilde{M}_N \left( \displaystyle\lim_{\varepsilon \to 0} \boldsymbol{q}^\varepsilon \right) = \boldsymbol{\widetilde{b}_N} $ by virtue of Lemma \ref{MnInvert}. Furthermore, we know that $\widetilde{M}_N \boldsymbol{e_1} = (r_0 \boldsymbol{v_1}^\top {{}\widetilde{A}_N^1}^{\top} \boldsymbol{e_1}, \cdots, r_0 \boldsymbol{v_{2^N-1}}^\top {{}\widetilde{A}_N^1}^\top \boldsymbol{e_1}, \boldsymbol{u_N}^\top \boldsymbol{e_1}) = \boldsymbol{e_{2^N}}$  because ${{}\widetilde{A}_N^1}^\top = {\widetilde{A}_N^1}$ and $\textnormal{Ker}({\widetilde{A}_N^1}) = \textnormal{Span}\left( \boldsymbol{e_1} \right)$. Finally,

\begin{equation} 
\lim_{\varepsilon \to 0} \boldsymbol{q}^\varepsilon = \widetilde{M}_N^{-1} \boldsymbol{\widetilde{b}_N} = \Phi {\widetilde{M}_N}^{-1} \boldsymbol{e_{2^N}} = \Phi \boldsymbol{e_1}.
\end{equation}
Then, since we have $\boldsymbol{\widetilde{b}_N} = \varepsilon \boldsymbol{\gamma} + \Phi \boldsymbol{e_{2^N}}$, with $\boldsymbol{\gamma} \in \textnormal{Span}\left( e_1,\dots,e_{2^N-1} \right)$, we have $\boldsymbol{q}^\varepsilon = \widetilde{M}_N(\varepsilon)^{-1} \boldsymbol{\widetilde{b}_N} = \varepsilon \widetilde{M}_N(\varepsilon)^{-1} \boldsymbol{\gamma} + \Phi \boldsymbol{e_1}$ and
\begin{eqnarray*}
\mathcal{E}(\boldsymbol{q}^\varepsilon, \boldsymbol{\xi}^\varepsilon) & = & (\boldsymbol{q}^\varepsilon)^\top A_N(\boldsymbol{\xi}^\varepsilon) \boldsymbol{q}^\varepsilon\\
& = &\left( \varepsilon \boldsymbol{\gamma}^\top \left(\widetilde{M}_N(\varepsilon)^{-1}\right)^\top + \Phi \boldsymbol{e_1}^\top \right) A_N(\varepsilon) \left( \varepsilon \widetilde{M}_N(\varepsilon)^{-1} \boldsymbol{\gamma }+ \Phi \boldsymbol{e_1} \right)\\
& = &\varepsilon^2 \boldsymbol{\gamma}^\top \left(\widetilde{M}_N(\varepsilon)^{-1}\right)^\top A_N(\varepsilon) \widetilde{M}_N(\varepsilon)^{-1} \boldsymbol{\gamma}
+ \varepsilon \Phi \boldsymbol{\gamma}^\top \left(\widetilde{M}_N(\varepsilon)^{-1}\right)^\top A_N(\varepsilon) \boldsymbol{e_1}\\
 & & + \varepsilon \Phi \boldsymbol{e_1}^\top A_N(\varepsilon) \widetilde{M}_N(\varepsilon)^{-1} \boldsymbol{\gamma}+ \Phi^2 \boldsymbol{e_1}^\top A_N(\varepsilon) \boldsymbol{e_1}
\end{eqnarray*}

Since $\displaystyle \lim_{\varepsilon \rightarrow 0} \left(\varepsilon A_N(\varepsilon) \right) = \widetilde{A}_N^1$ and $\displaystyle \lim_{\varepsilon \to 0} \widetilde{M}_{N}(\varepsilon) = 
\widetilde{M}_{N}$, then one successively has

\begin{equation}
\begin{array}{lll}
\varepsilon^2 \boldsymbol{\gamma}^\top \left(\widetilde{M}_N(\varepsilon)^{-1}\right)^\top A_N(\varepsilon) \widetilde{M}_N(\varepsilon)^{-1} \boldsymbol{\gamma}& \underset{\varepsilon \rightarrow 0}{\longrightarrow} & 0\\
\varepsilon \Phi \boldsymbol{\gamma}^\top \left(\widetilde{M_N}(\varepsilon)^{-1}\right)^\top A_N(\varepsilon) \boldsymbol{e_1}&\underset{\varepsilon \rightarrow 0}{\longrightarrow}& 0\\
\varepsilon \Phi \boldsymbol{e_1}^\top A_N(\varepsilon) \widetilde{M}_N(\varepsilon)^{-1} \boldsymbol{\gamma}&\underset{\varepsilon \rightarrow 0}{\longrightarrow}& 0 \\
\Phi^2 \boldsymbol{e_1}^\top A_N(\varepsilon) \boldsymbol{e_1} &\underset{\varepsilon \rightarrow 0}{\longrightarrow} & \left(1+ \frac{N^2}{\Lambda - 1}\right)r_0 \Phi^2,\\
\end{array}
\end{equation}
because $\widetilde{A}_N^1 \boldsymbol{e_1} = 0$, $\boldsymbol{e_1}^\top {{}\widetilde{A}_N^1}^\top = 0$ and $\boldsymbol{e_1}^\top A_N(\varepsilon) \boldsymbol{e_1} = r_0 + \frac{N r_0}{\frac{\Lambda - 1}{N} - \left(\frac{2^{N+1}-2}{N}-1\right) \varepsilon}$.

Hence, 
$$
\lim_{\varepsilon \to 0} \mathcal{E}(\boldsymbol{q}^\varepsilon, \boldsymbol{\xi}^\varepsilon)
= r_0 \Phi^2 \left( \frac{N^2}{\Lambda - 1} + 1 \right) = m,
$$
which concludes the proof.

\section{Case of a fluid driven by Navier-Stokes equations: some numerical results}

As stated in the introduction, we study in this section an infinite dimensional optimization problem. Now the unknown is the whole shape of the dyadic tree and we do not anymore neglect the connections at bifurcations nor do we constrain the branches to remain cylindrical. Moreover, we will now use the full non linear Navier-Stokes equations and still minimize the dissipated energy of the fluid. Our goal is to compare the shapes obtained in this case by numerical means to that obtained theoretically for Poiseuille's laws in the previous sections. Nevertheless, we will not be able to catch the behaviour of the case with identical pressures imposed at each outlets since it is unstable. It is not possible in our numerical simulations to make the pressure exactly equal because of the rounding errors and the local meshes. Only a dedicated algorithm that would include the symmetry and prevent the closing of branches would be able to catch this behaviour. Consequently, from the theoretical study, we expect to observe the closing of every branches except one.

\par We focus here on the 3D case but the 2D case can be easily adapted from this 3D study and numerical results will be presented both in 2D and 3D.

\par The partial differential equations describing the behaviour of the fluid and the boundary conditions is now more general.
In particular, this model is a convenient choice for the modeling of the bronchial tree (see for instance \cite{mauroy-3DHydro, mauroy-interplay, mauroy-optimalTreeDangerous, mauryMeunier}).
\par Let us recall some classical definitions in incompressible Fluid Mechanics.
\begin{dfntn}
Let $\boldsymbol{u}$ be a smooth vector field of $\mathbb{R}^3$ (for instance $C^1$). We define
\begin{enumerate}
\item the stretching tensor of $\boldsymbol{u}$ (symmetric part of the gradient tensor):
$$
\varepsilon (\boldsymbol{u})= \left(\frac{1}{2}\left(\frac{\partial
u_i}{\partial x_j}+\frac{\partial u_j}{\partial
x_i}\right)\right)_{1\leq i,j \leq 3}.
$$
\item the doubly contracted product of two stretching tensors $\varepsilon (\boldsymbol{u})$ and $\varepsilon (\boldsymbol{v})$:
$$
\varepsilon (\boldsymbol{u}):\varepsilon (\boldsymbol{v})= \frac{1}{4}\sum_{i=1}^3\sum_{j=1}^3\left(\frac{\partial u_i}{\partial x_j}+\frac{\partial u_j}{\partial x_i}\right)\left(\frac{\partial v_i}{\partial x_j}+\frac{\partial v_j}{\partial x_i}\right).
$$
\item the stress tensor of $(\boldsymbol{u},p)$, where $\boldsymbol{u}$ is a smooth vector field of $\mathbb{R}^3$ representing the velocity of a fluid whose viscosity is $\mu>0$ and $p$ a function representing the pressure defined on $\mathbb{R}^3$:
$$
\sigma (\boldsymbol{u},p)=-pI_3+2\mu \varepsilon (\boldsymbol{u}),
$$
where $I_3$ is the identity tensor of $\mathbb{R}^3$.
\end{enumerate}
\end{dfntn}
\subsection{The shape optimization problem}\label{pres}
We now give some precisions on the frame of our study. Let $\Omega$ be a generic three dimensional $Y$-shaped domain. We will denote by $\partial\Omega$ the boundary of $\Omega$. In the sequel, we will assume that the inlet $E$ of $\Omega$ is a disk and that the outlet $S$ consists in two identical disks. We decompose the boundary of $\Omega$ as the disjoint union $\partial \Omega =E\cup \Gamma\cup S$. $\Gamma$, the lateral boundary, is the main unknown or the shape we want to determine. In the following, we will thus consider that $E$ and $S$ are fixed part of the boundary.
\par We assume that a fluid driven by the Navier-Stokes equations flows in $\Omega$. In particular, this model is convenient to represent the upper part of the bronchial tree since the velocity of the air at the beginning of the trachea is important enough to consider that the flow is inertial or even turbulent (the Reynolds number in trachea is about $1000$ at rest). 
Then, the velocity $\boldsymbol{u}:\Omega \rightarrow \mathbb{R}^3$ and the pressure $p:\Omega\rightarrow \mathbb{R}$, are solutions of the Navier-Stokes system
\begin{equation}\label{NS1}
\left\{\begin{array}{ll}
\displaystyle -\mu \Delta \boldsymbol{u}+\nabla p+\rho \boldsymbol{u}\cdot\nabla \boldsymbol{u}=0 & x\in \Omega \\
\displaystyle \nabla\cdot \boldsymbol{u}=0 & x\in \Omega \\
\boldsymbol{u}=\boldsymbol{u_0} & x\in E \\
\boldsymbol{u}=0 & x\in \Gamma \\
-p\boldsymbol{n}+2\mu \varepsilon (\boldsymbol{u})\boldsymbol{n}=\boldsymbol{h} & x\in S,
\end{array}
\right.
\end{equation}
where
\begin{itemize}
\item $\boldsymbol{n}$ denotes the outward-pointing unit normal vector at a given point of the boundary $\partial\Omega$,
\item $\boldsymbol{u_0}$ is a parabolic velocity profile (\textit{i.e.}~a Poiseuille's flow is imposed at the inlet $E$), that is
$$
\boldsymbol{u_0}=^\top\!\!(0,0,c(r^2-R^2)),
$$
where $c$ is a negative constant so that the flow is ingoing, and $R$ the radius of the inlet.
\end{itemize}

\begin{rmrk}
Let us clarify the choice of the boundary conditions for this model. 
\begin{itemize}
\item The condition $\boldsymbol{u}=0$ on $\Gamma$ is the so-called {\it no-slip} boundary condition and means that the fluid ``sticks'' to the wall.
\item Practically speaking, we will impose $\boldsymbol{h}=-p_0\boldsymbol{n}$ at the outlet $S$. This condition comes more or less from the assumption that the pressure is the sole force acting on $S$ and that it is known and equal to $p_0$. Drawing a parallel with the modeling of the bronchial tree, these conditions simulate muscles applying the same force per unit of surface for each outlet ($-p_0 \boldsymbol{n}$) and thus using the same energy for each (see \cite{mauryMeunier}). Similar boundary conditions are more detailed in \cite[chapter 5]{BF} and in \cite{bruneau,bruneau2}.\\
\end{itemize}
\end{rmrk}
Notice that the classical theory of Navier-Stokes equations (see \cite{galdi,temam}) gives information on the existence and uniqueness of the solution of System (\ref{NS1}):
\begin{thrm}
\label{thmNS}
Let $\Omega$ be a bounded Lipschitz domain of $\mathbb{R}^3$. Let us assume that $\boldsymbol{u_0}$ belongs to the Sobolev space $(H^{{3/2}}(E))^3$ and $\boldsymbol{h}$ belongs to $(H^{{1/2}}(S))^3$. There exists $A>0$ such that if the
viscosity $\mu$ is larger than $A$, then Problem (\ref{NS1}) has a
unique solution $(\boldsymbol{u},p)\in (H^1 (\Omega ))^3\times L^{2}(\Omega )$.
\end{thrm}
Let us now introduce the shape optimization problem we want to solve. The objective functional $J(\Omega)$ is the energy dissipated by the fluid, \textit{i.e.}
\begin{equation}
J(\Omega )=2\mu \int_\Omega |\varepsilon
(\boldsymbol{u})|^2\d x 
\end{equation}
To make the statement of the optimization problem precise, we need to define a class of admissible shapes. As classically in shape optimization, we fix the measure $V_0>0$ of $\Omega$. Let us introduce
\begin{equation}\label{class1}
\mathcal{O}_{V_0} = \left\{\Omega
\textnormal{ bounded and simply connected domain in } \mathbb{R} ^3\textnormal{ containing $E$ and $S$, such that }
\textrm{meas}(\Omega)= V_0\right\}.
\end{equation}
The shape optimization problem writes
\begin{equation}\label{pbOptim1}
\left\{\begin{array}{l}
\min J(\Omega)\\
\Omega\in \mathcal{O}_{V_0}.
\end{array}\right.
\end{equation}
The question of knowing if Problem (\ref{pbOptim1}) has or not a solution is still an open problem. Nevertheless, it is possible to show that a problem, very close to Problem (\ref{pbOptim1}) but a little bit constrained, has a solution.
\par Restricting the set of admissible shapes is a very common approach in shape optimization, since these problems are often ill-posed (see for instance \cite{allaire,HP}). A very close existence theorem is announced in \cite{HPri} and proved in \cite{Henrot-Privat}, considering instead of the set $\mathcal{O}_{V_0}$, a set of domains verifying an $\varepsilon$-cone property, which yields some kind of uniform regularity (see \cite{Che1,HP,DZ} for some reminders on the $\varepsilon$-cone property and its consequences on the existence of optimal shapes).

\subsection{Computation of the shape derivative of $J$}
The classical algorithm we will recall in Appendix \ref{augLagAlg} is based on the use of the well known shape derivative.
We recall here the expression of such a derivative. The details of its calculation are given in \cite{Henrot-Privat}.
\par It is possible to define an adjoint state for System (\ref{NS1}) that is useful to write the shape derivative in a very usual form (see \cite[Theorem 5.9.2]{HP}). Let $(\boldsymbol{v},q)$ be solution (in the case where it exists) of the linearized Navier-Stokes system
\begin{equation}
\left\{\begin{array}{ll}\label{PAdjoinNS}
\displaystyle -\mu \Delta \boldsymbol{v} + (\nabla \boldsymbol{u})^\top\!\! \ \boldsymbol{v} - (\nabla \boldsymbol{v}) \ u +\nabla q =-2\mu \Delta \boldsymbol{u} & x\in \Omega \\
\nabla\cdot \boldsymbol{v}=0 & x\in \Omega \\
\boldsymbol{v}=\boldsymbol{0} & x\in E\cup \Gamma \\
-q\boldsymbol{n}+2\mu \varepsilon (\boldsymbol{v})\boldsymbol{n}+(\boldsymbol{u}\cdot \boldsymbol{n})\boldsymbol{v}-4\mu \varepsilon (\boldsymbol{u}) \boldsymbol{n}=0 & x\in S. 
\end{array}
\right.
\end{equation}
The following existence and uniqueness result is established in \cite[Proposition 3.1]{Henrot-Privat}.
\begin{prpstn}\label{ExistenceRegulariteAdjoint}
Let $\Omega$ be a bounded Lipschitz domain of $\mathbb{R}^3$. There exists $A > 0$ such that, if the viscosity $\mu$ is larger than $A$, then the problem
(\ref{PAdjoinNS}) has a unique solution $(\boldsymbol{v},q)$. Moreover,
this solution belongs at least to $(H^1 (\Omega))^3\times L^{2}(\Omega)$.
\end{prpstn}
Notice that the restriction in the size of the viscosity $\mu$ ensures the existence, the uniqueness and a sufficient regularity of $\boldsymbol{u}$, solution of the Navier-Stokes equation \ref{NS1} (see theorem \ref{thmNS}).

\par We are now able to define the derivative of $J$ with respect to the domain. Let us consider a regular vector field $\boldsymbol{ V}:\mathbb{R}^3\to\mathbb{R}^3$ with compact support which does not meet neither $E$ nor $S$. For small $t$, we
define  $\Omega _t=(I+t\boldsymbol{ V})\Omega$, the image of $\Omega$ by a
perturbation of identity and $f(t)=  J(\Omega _t)$. We recall that
the shape derivative of $J$ at $\Omega$ with respect to $\boldsymbol{ V}$
is $f'(0)$. We will denote it by $\d J(\Omega;\boldsymbol{ V})$. To compute
it, we first need to compute the derivative of the state equation.
We use here the classical results on shape derivatives as in
\cite{HP}, \cite{Mu-Si}, \cite{So-Zo}. The derivative of
$(\boldsymbol{u},p)$ is the solution of the linear system
\begin{equation}
\left\{\begin{array}{ll} \label{PDérivé}
\displaystyle -\mu \Delta \boldsymbol{u'}+\boldsymbol{u'}\cdot\nabla \boldsymbol{u}+\boldsymbol{u}\cdot\nabla \boldsymbol{u'}+\nabla p'=0 & x\in \Omega \\
\nabla\cdot \boldsymbol{u'}=0 & x\in \Omega \\
\boldsymbol{u'}=\boldsymbol{0} & x\in E \\
\displaystyle  \boldsymbol{u'}=-\frac{\partial \boldsymbol{u}}{\partial \boldsymbol{n}}
(\boldsymbol{V}\cdot \boldsymbol{n}) & x\in \Gamma \\
\displaystyle -p'\boldsymbol{n}+2\mu \varepsilon (\boldsymbol{u'})\boldsymbol{n}=0 & x\in S,
\end{array}
\right.
\end{equation}
where $\boldsymbol{u}$ is, under assumption that $\mu$ is large enough, the unique solution of (\ref{NS1}). Similar formula in the context of Shape Optimization in Fluid Mechanics are established for instance in \cite{mo-pi,Piro,simon1,simon2}.
\par Now, we have (see \cite{HP}, \cite{So-Zo})
\begin{equation}\label{deriv1}
    \d J (\Omega ; \boldsymbol{V})=4\mu \int_\Omega \varepsilon
(\boldsymbol{u}):\varepsilon (\boldsymbol{u'})\d x+2\mu \int_\Gamma
|\varepsilon (\boldsymbol{u})|^2(\boldsymbol{ V}\cdot \boldsymbol{n})\d s.
\end{equation}
Using the adjoint state (\ref{PAdjoinNS}), it is possible to rewrite $\d J (\Omega , \boldsymbol{V})$ in a more workable form (see \cite[Proposition 3.2]{Henrot-Privat}):
\begin{equation}
\label{DeriveeCritere} \d J (\Omega ; \boldsymbol{V})=2\mu \int_\Gamma
\left(\varepsilon (\boldsymbol{u}):\varepsilon (\boldsymbol{v})-|\varepsilon
(\boldsymbol{u})|^2\right)(\boldsymbol{V}\cdot\boldsymbol{n})\d s .
\end{equation}\
We will now introduce the shape gradient of the criterion $J$ as a functional defined on the boundary such that $\d J (\Omega ; \boldsymbol{V})=\int_{\Gamma}\nabla J(\Omega)\cdot \boldsymbol{V}\d s$, in other words
\begin{equation}
\label{shapegrad}
\nabla J (\Omega)=2\mu \left(\varepsilon (\boldsymbol{u}):\varepsilon (\boldsymbol{v})-|\varepsilon
(\boldsymbol{u})|^2\right)\boldsymbol{n}.
\end{equation}

Thanks to this expression of the shape gradient, we are now able to implement an augmented Lagrangian algorithm in order to solve Problem (\ref{pbOptim1}). The algorithm is classical and is described in Appendix \ref{augLagAlg}.

\subsection{Numerical results}\label{numResults}

In this section, we present 2D and 3D numerical simulations using the augmented Lagrangian algorithm described in Appendix \ref{augLagAlg} based on the shape gradient (\ref{shapegrad}). They have been implemented with the script of the software {\sl Comsol Multiphysics}. The 2D algorithm is an easy adaptation of the 3D algorithm.

The Navier-Stokes systems and the adjoint state are solved with a direct finite elements method using Lagrange elements $P1$ for pressures and $P2$ for velocities. The displacements of the mesh are $P1$. The multifrontal package (UMFPACK) is used to solve the resultant linear system and a modified Newton like method is used to treat the non linear term. At each iteration $k$, each node of the geometry is perturbed by the discretized operator $(I+\varepsilon_k\boldsymbol{d_k})(\Omega_k)$ (classical Arbitrary Lagrangian Eulerian method) and it is necessary to remesh the geometry when the displacement of the mesh becomes too large.

The Reynolds number in the following computations is $100$ which is not only large enough to observe inertial effects near the bifurcation, but small enough to let the computations run in a reasonable amount of time. The effect of inertia on flow distribution in the bifurcation can be seen on Figure \ref{inertia} where the image on the left represents a non inertial flow (Reynolds $0$) and the image on the right an inertial flow with Reynolds $100$.

\begin{figure}[!h]
\begin{center}
\includegraphics[height=4cm, angle=-90, clip=true, trim=10cm 20cm 10cm 5cm]{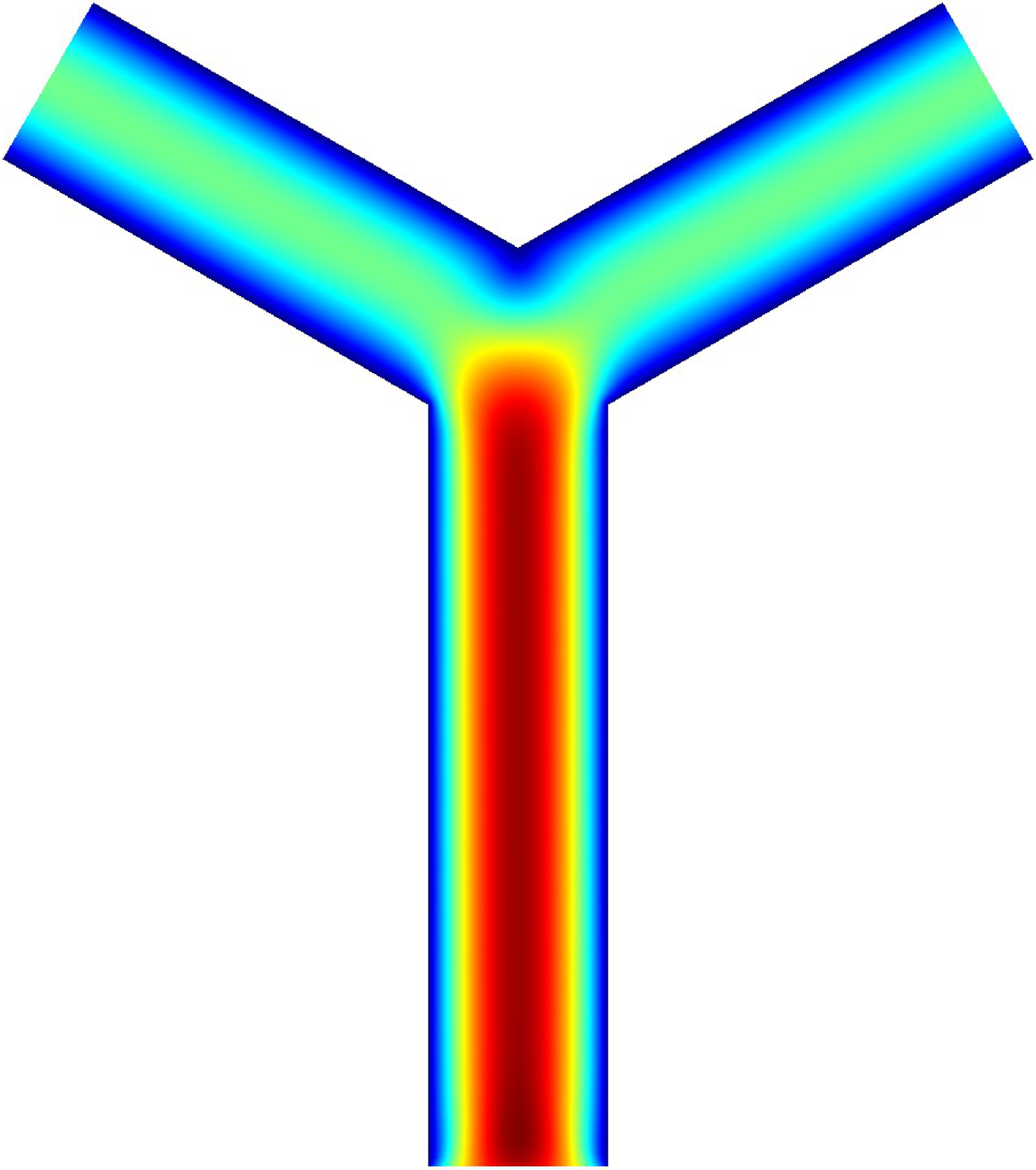}
\includegraphics[height=4cm, angle=-90, clip=true, trim=10cm 20cm 10cm 5cm]{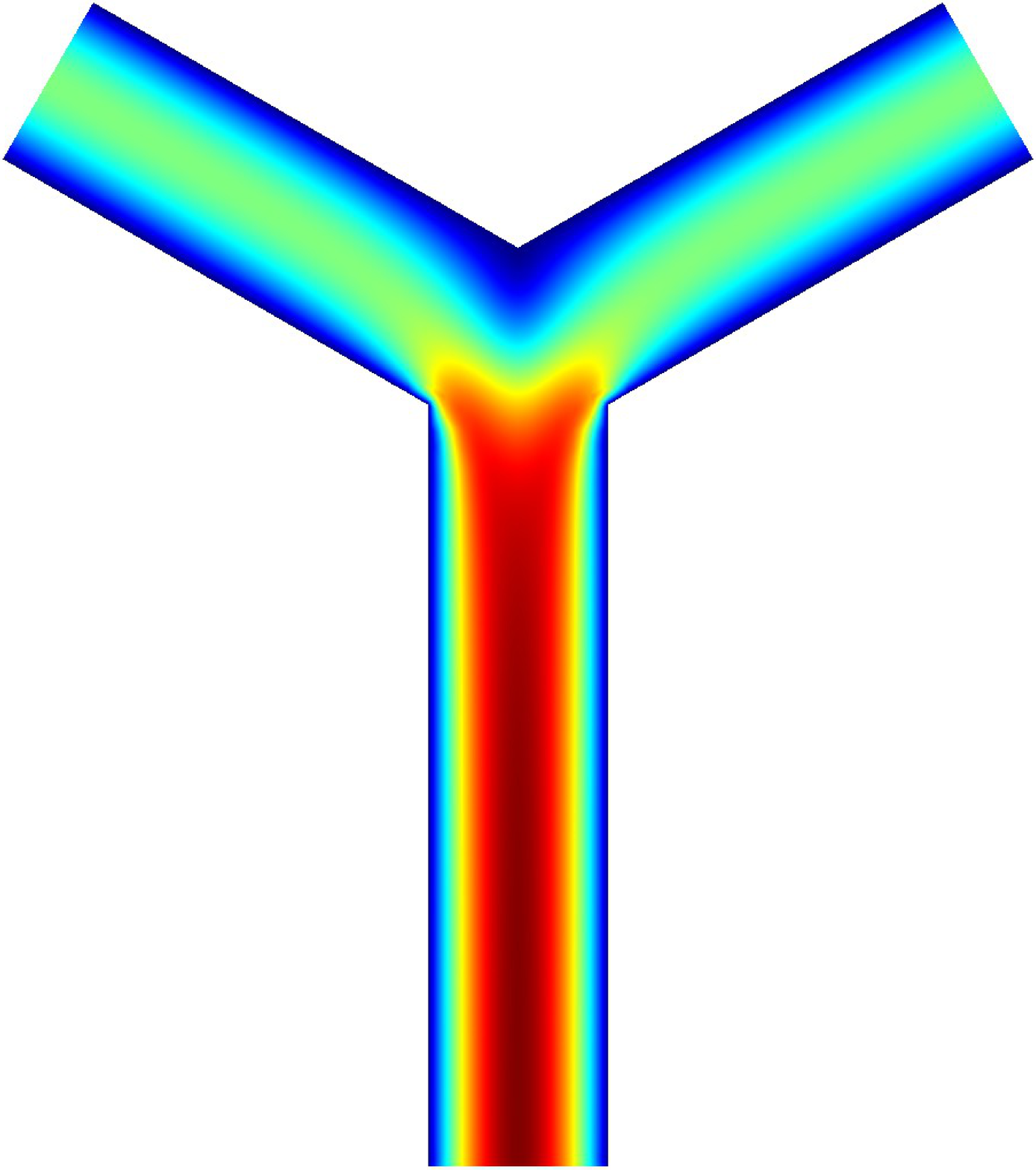}
\caption{Velocity norm (jet colors mapping) in a 2D bifurcation for a non inertial flow with Reynolds $0$ (left) and for an inertial flow with Reynolds $100$ (right). Inertia alters the flow mostly around the bifurcation.}
\label{inertia}
\end{center}
\end{figure}

We normalized the applications $J(\cdot)$ and $V(\cdot)$ in such a way that $J(\Omega_0)=1$ and $V(\Omega_0)=1$. 

\subsubsection{1st Case: Neumann conditions are imposed at the inlet and Dirichlet conditions are imposed at the outlets}
\label{vconnuNS}

In this section, the initial geometry $\Omega_0$ is a bifurcation whose branches are identical except for the length of the mother branch which is $1.5$ the length of the daughter branches, see Figure \ref{numres12D} (2D, left image) and Figure \ref{numres1} (3D, left image). The length of the mother branch has been chosen longer to avoid that the flow near the bifurcation interacts too much with the inlet flow since the algorithm tends to shorten the mother branch. Practically, a mother branch too short leads to convergence problem. The angle between two nearby branches is $\pi/3$.

The boundary conditions correspond to that of Case \ref{vconnu} adapted to the non-linear regime. Thus a parabolic velocity profile is imposed at each outlets (Dirichlet condition) and Neumann boundary conditions $\sigma(p,\boldsymbol{u})\boldsymbol{n}=-p_0\boldsymbol{n}$ are imposed at the inlet, which is ``formally'' interpreted as a pressure constraint since viscous effects are negligible. The pressure imposed at the inlet is $0$.

Figures \ref{numconv12D} (2D) and \ref{numconv1} (3D) show the convergence of the different quantities of the problem: the Lagrange multiplier (upper left), the Lagrangian function (upper right), the viscous energy (lower left) and the volume (lower right). Except the Lagrange multiplier, these quantities have been normalized. As expected, the curves are oscillating around their asymptotic values since verifying the volume constraint and minimizing the criterion are advantaged one after the other by the algorithm. This ``advantage'' is controlled by the Lagrange multiplier. Hence when it is too large the volume constraint is stronger and the algorithm authorized an increase of the criterion; if it is too small then the algorithm let loose the constraint and decrease the criterion. The amplitude of these oscillations decreases with the iterations.

In 2D, the mesh consists in $9296$ triangle elements and no remeshing was needed to ensure convergence. Convergence was reached for $1000$ iterations of the algorithm in about 6 hours using 2 cores of a Xeon processor (2.33 GHz). The viscous energy dissipated in the bifurcation, which is the quantity minimized by our algorithm, has been reduced by $20.4 \%$ in the final geometry in comparison with its value in the initial geometry. Volume precision at convergence was smaller than $2 \times 10^{-6}$.

In 3D, the mesh consists in $17959$ tetrahedral elements and no remeshing was needed either. Convergence was reached for $1700$ steps. Computation time was about $46$ hours using 2 cores of a Xeon processor (2.33 GHz). The viscous energy dissipated in the bifurcation has been reduced by $22.5 \%$ between the initial and final geometry. The volume constraint at the end of the algorithm was smaller than $3 \times 10^{-5}$.

\begin{figure}[!h]
\begin{center}
\includegraphics[height=5cm]{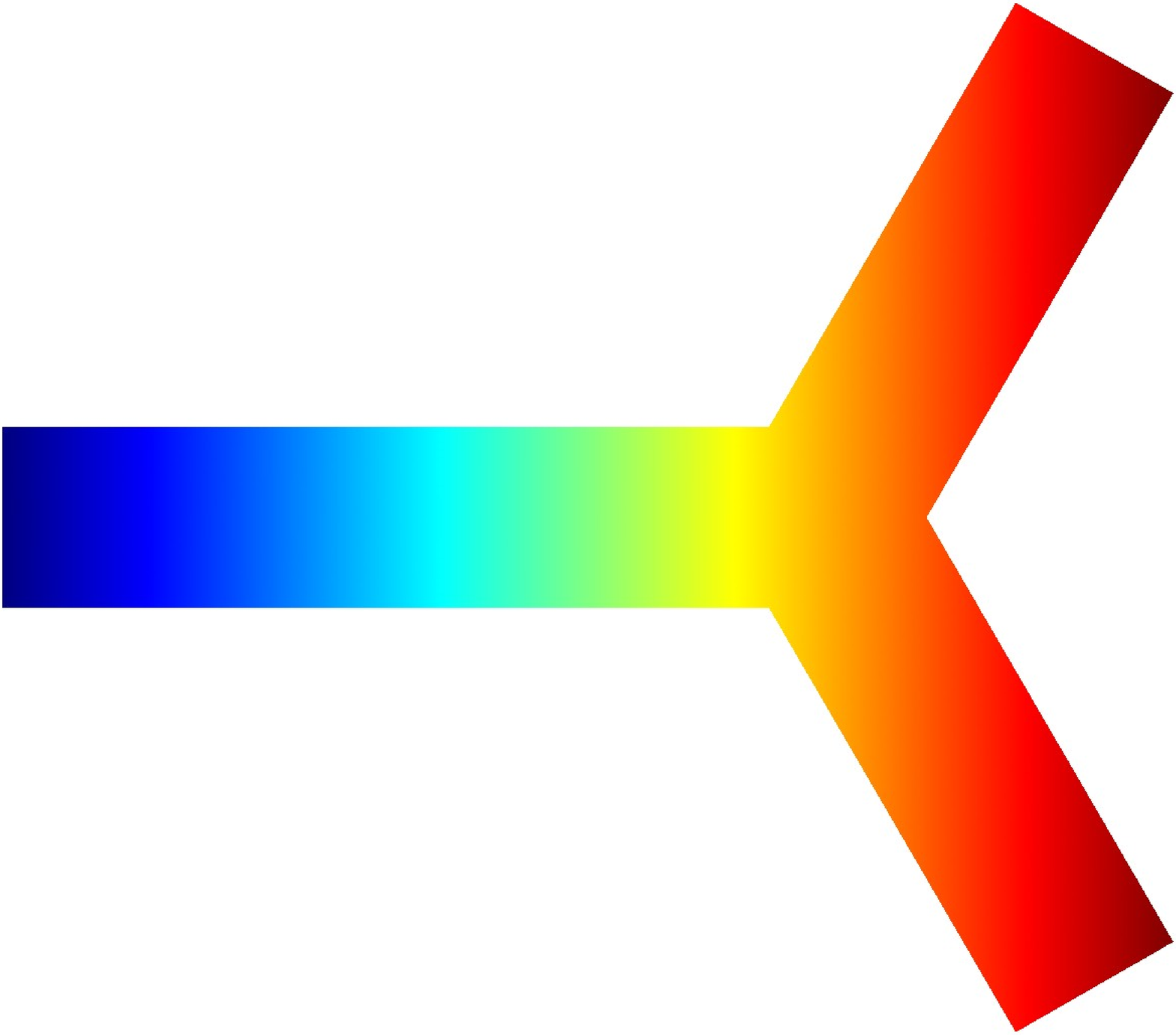}
\includegraphics[height=5cm]{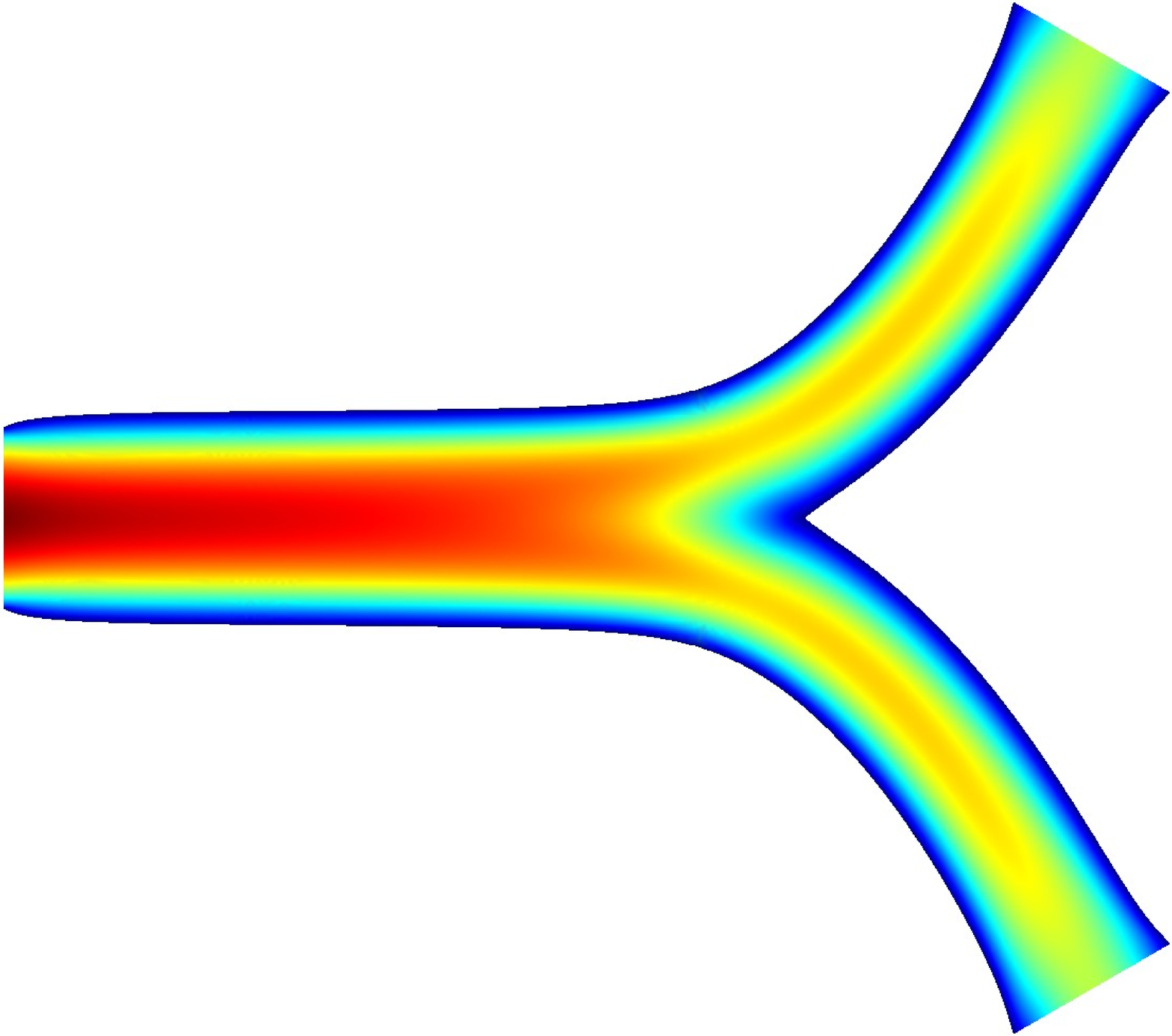}
\caption{2D Case. Left: initial geometry $\Omega_0$, the inlet is the blue branch while the outlets are the red branches. Right: final shape reached by the algorithm, the colors represent the norm of the velocity $u$ (increasing with the warmth of colors).}
\label{numres12D}
\end{center}
\end{figure}

\begin{figure}[!h]
\begin{center}
\includegraphics[height=8cm]{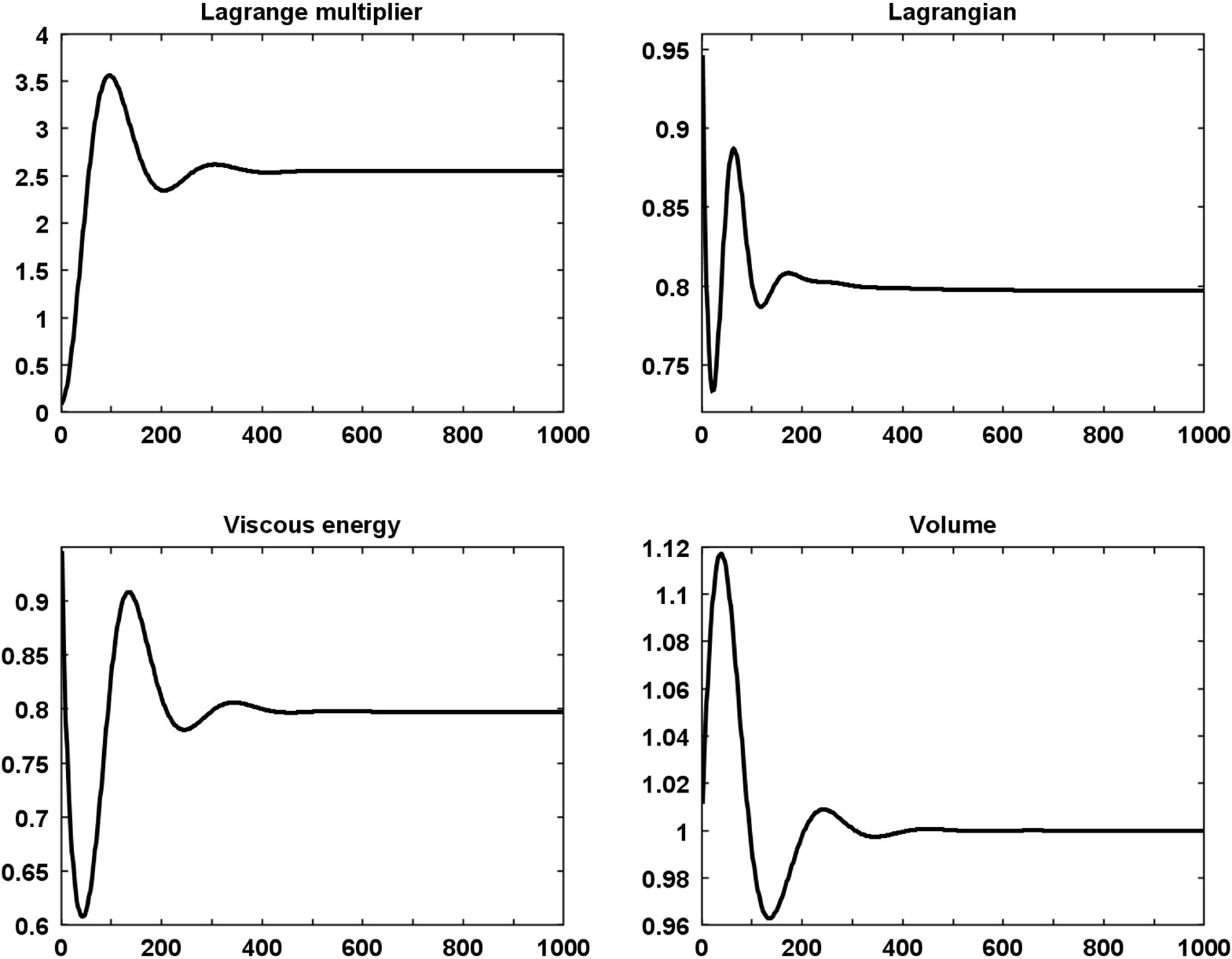}
\caption{2D case. Convergence curves of the Lagrange multiplier $\mu_k$, the Lagrangian $\mathcal{L}_b(\Omega_k,\mu_k)$, the viscous energy $J(\Omega_k)$ and the volume $V(\Omega_k)$. The x-coordinates represent the iterations number.}
\label{numconv12D}
\end{center}
\end{figure}

\begin{figure}[!h]
\begin{center}
\includegraphics[height=5cm]{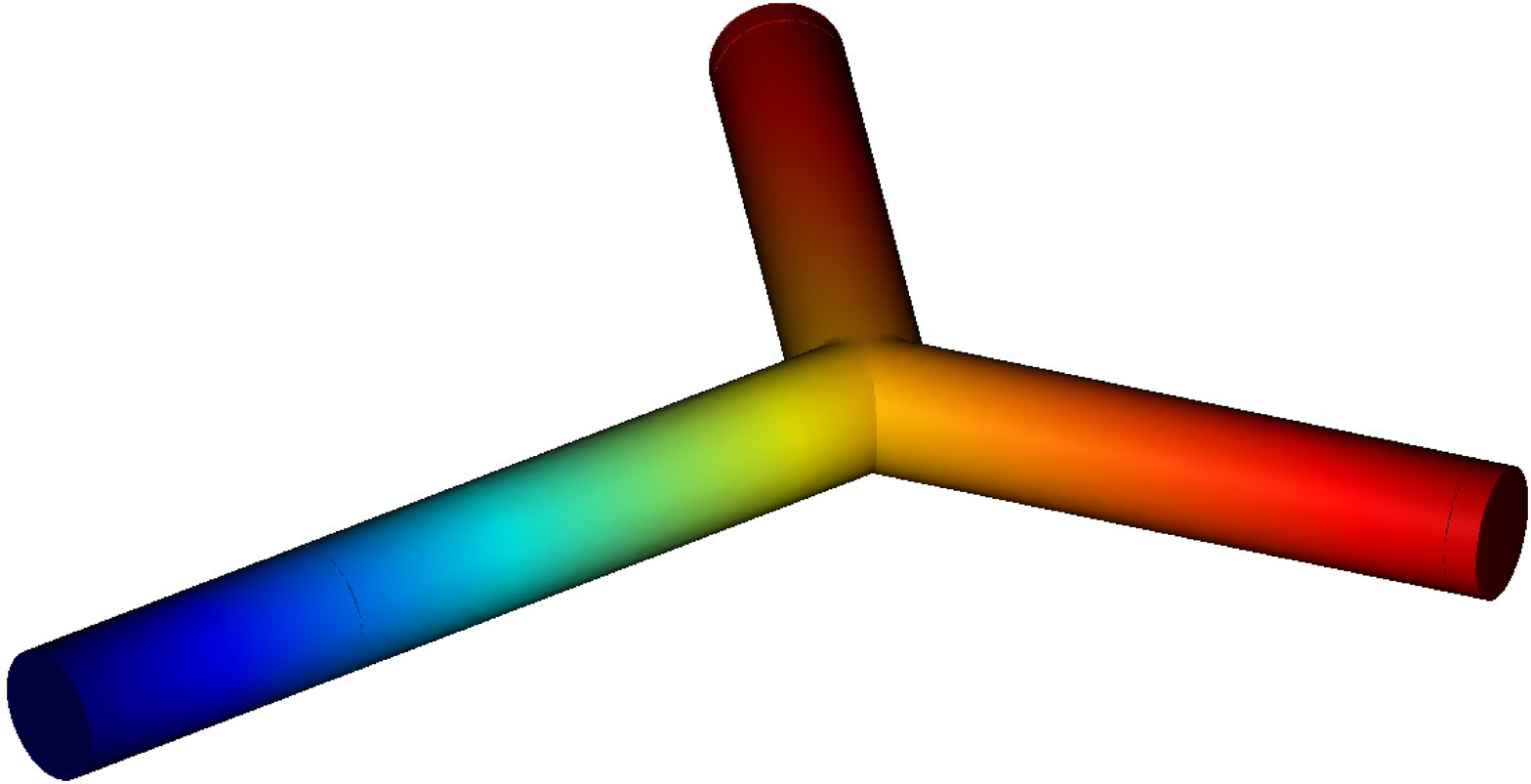}
\includegraphics[height=5cm]{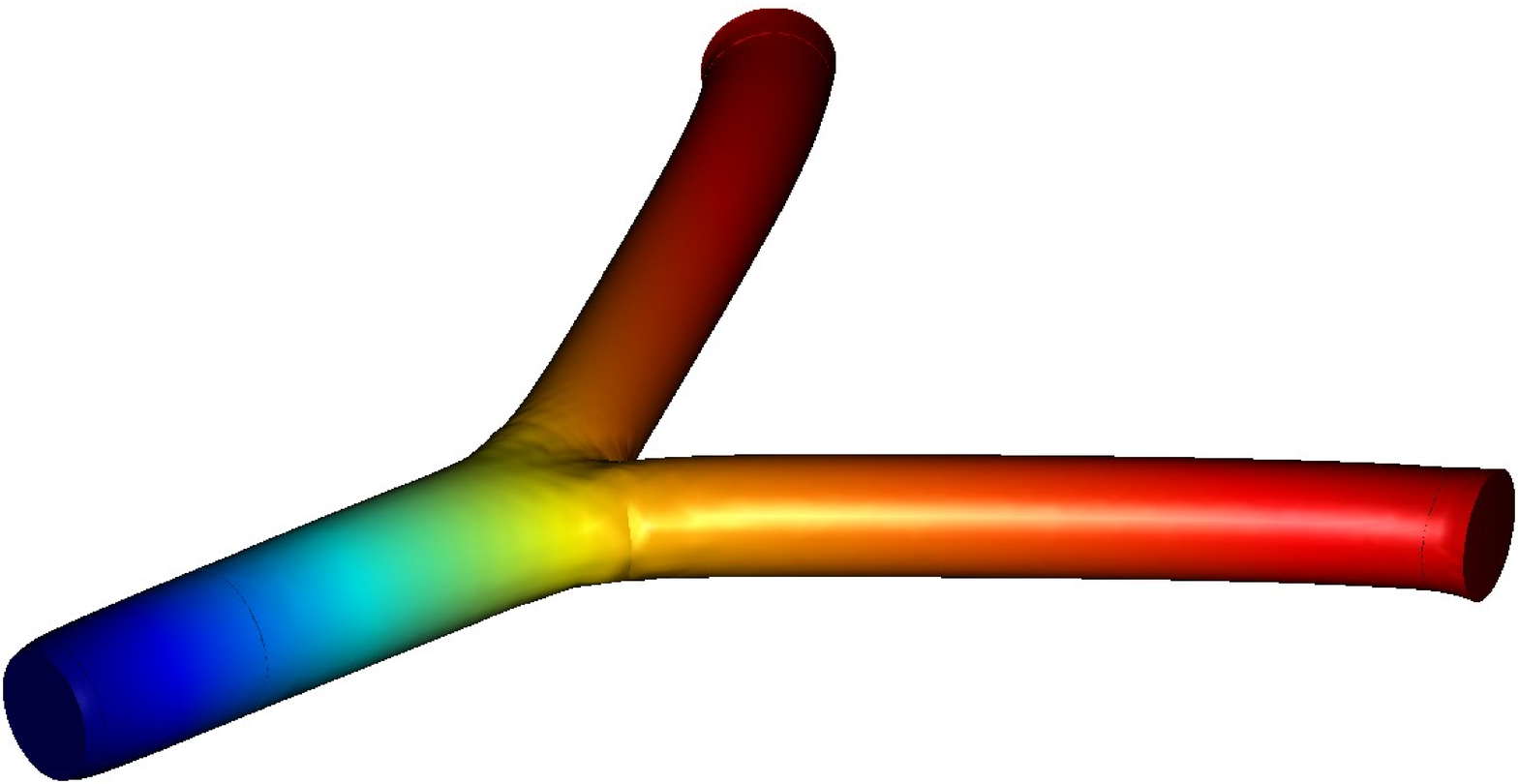}
\caption{3D case. Left: initial geometry $\Omega_0$, the inlet is the blue branch while the outlets are the red branches. Right: final shape reached by the algorithm.}
\label{numres1}
\end{center}
\end{figure}

\begin{figure}[!h]
\begin{center}
\includegraphics[height=8cm]{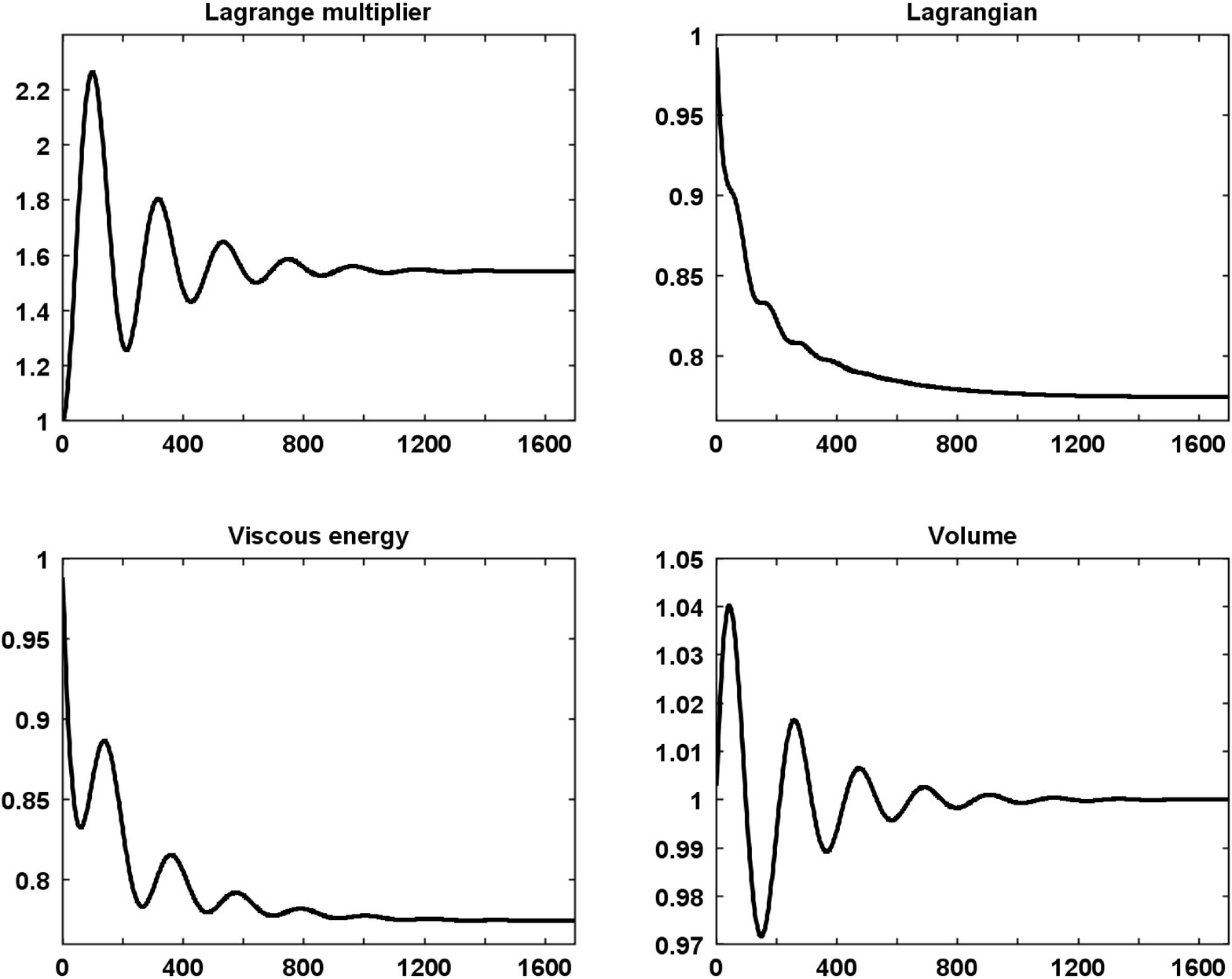}
\caption{3D case. Convergence curves of the Lagrange multiplier $\mu_k$, the Lagrangian $\mathcal{L}_b(\Omega_k,\mu_k)$, the viscous energy $J(\Omega_k)$ and the volume $V(\Omega_k)$. The x-coordinates represent the iterations number.}
\label{numconv1}
\end{center}
\end{figure}\ \\

\subsubsection{2nd case: Dirichlet conditions are imposed at the inlet and Neumann conditions are imposed at the outlets}
\label{pconnuNS}

In this section, the initial geometry $\Omega_0$ is a bifurcation whose branches are all identical, see Figure \ref{numres22D} (2D, left image) and \ref{numres2} (3D, left image). The angle between two nearby branches is $\pi/3$.

The boundary conditions correspond to that of Case \ref{pconnu} adapted to the non-linear regime. Thus a parabolic velocity profile is imposed at the inlet (Dirichlet condition) and Neumann boundary conditions $\sigma(p,\boldsymbol{u})\boldsymbol{n}=-p_0\boldsymbol{n}$ are imposed at the outlets. The pressures imposed at the two outlets are slightly different: $0 \ Pa$ on one outlet and $2 \times 10^{-4} \ Pa$ on the other.

As for the first case, the curves plotted on figure \ref{numconv22D} (2D) and \ref{numconv2} (3D) represent the convergence of the different quantities of the problem: the Lagrange multiplier (upper left), the Lagrangian function (upper right), the viscous energy (lower left) and the volume (lower right). Except the Lagrange multiplier, these quantities are still normalized. As before, they are oscillating around their convergence value (see 1st case).

In 2D, the mesh consists in 8624 triangle elements and no remeshing was needed. The convergence was achieved in $2800$ steps for a computing time of $8$ hours on 2 cores of a Xeon processor (2.33 GHz). In the final geometry the viscous dissipated energy is reduced by $54.5 \%$ relatively to the initial bifurcation. The precision of the volume constraint is smaller than $7 \times 10^{-5}$.

In 3D, convergence was achieved in $7000$ steps. A remeshing was performed at the step $4500$ (represented by the dotted vertical line on figure \ref{numres2}). The initial mesh consisted in $10044$ tetrahedral elements and the second mesh was finer with $37167$ tetrahedral elements. The computation was much slower after the remeshing because of the increased number of elements. The total time needed by the algorithm to converge was of $259$ hours on 2 cores of a Xeon processor (2.33 GHz). The viscous energy dissipated in the final geometry is reduced by $53.6 \%$ relatively to the viscous energy dissipated in the initial geometry. The volume constraint precision at convergence was smaller than $7 \times 10^{-5}$.

Similarly than the result found theoretically at low regime in Section \ref{pconnu}, we observe the closing of one branch. This indicates that this result should be more generally true and that it can probably be extended to non linear regime. A zoom of the 2D bifurcation is plotted on Figure \ref{zoom2D}, the arrows represents the normalized velocities and a region of fluid recirculation appears at the beginning of the closing branch. The colors on Figure \ref{zoom2D} represents the pressure which is almost constant along the branch that is closing (up). This induces a very small flow inside it (the flow and the branch diameter are decreasing together when the optimization algorithm is progressing).

\begin{figure}[!h]
\begin{center}
\includegraphics[height=5cm]{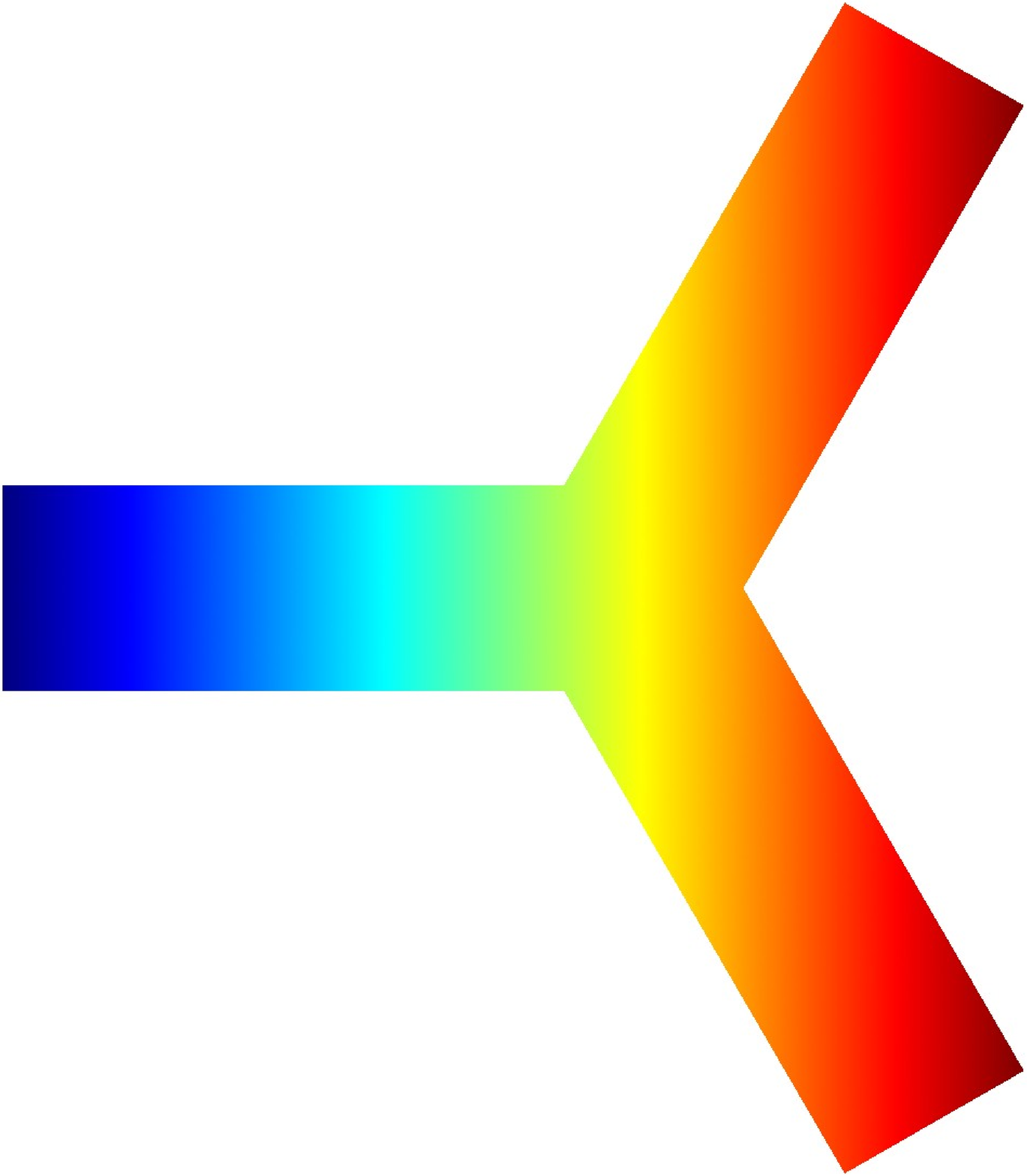}
\includegraphics[height=5cm]{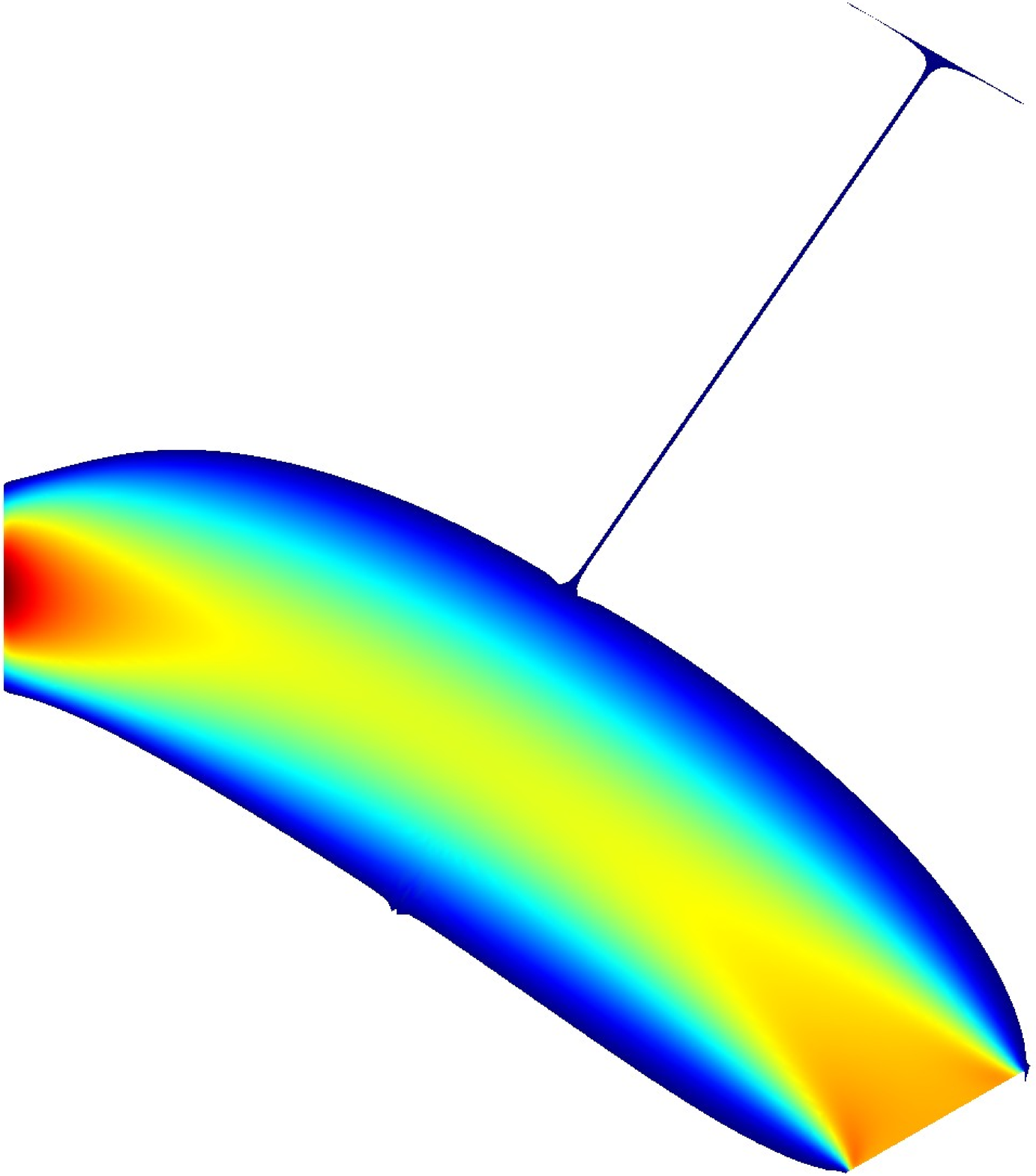}
\caption{2D case. Left: initial geometry $\Omega_0$, the inlet is the blue branch while the outlets are the red branches. Right: final shape reached by the algorithm, the color represents the norm of the velocity $u$ (increasing with the warmth of colors).}
\label{numres22D}
\end{center}
\end{figure}

\begin{figure}[!h]
\begin{center}
\includegraphics[height=5cm]{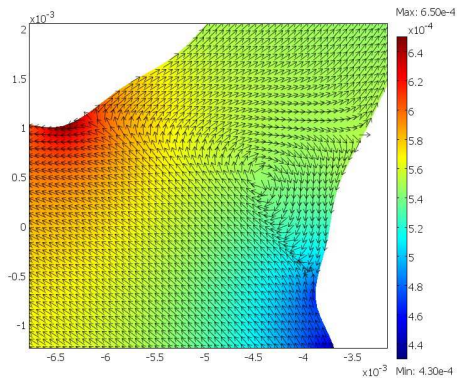}
\caption{Zoom on the bifurcation, the color represents the pressure and the arrows the directions of the velocities (normalized vector field). A recirculation zone is present in the bifurcation and the pressure is almost constant in the closing branch (up) inducing a very small flow in it.}
\label{zoom2D}
\end{center}
\end{figure}

\begin{figure}[!h]
\begin{center}
\includegraphics[height=8cm]{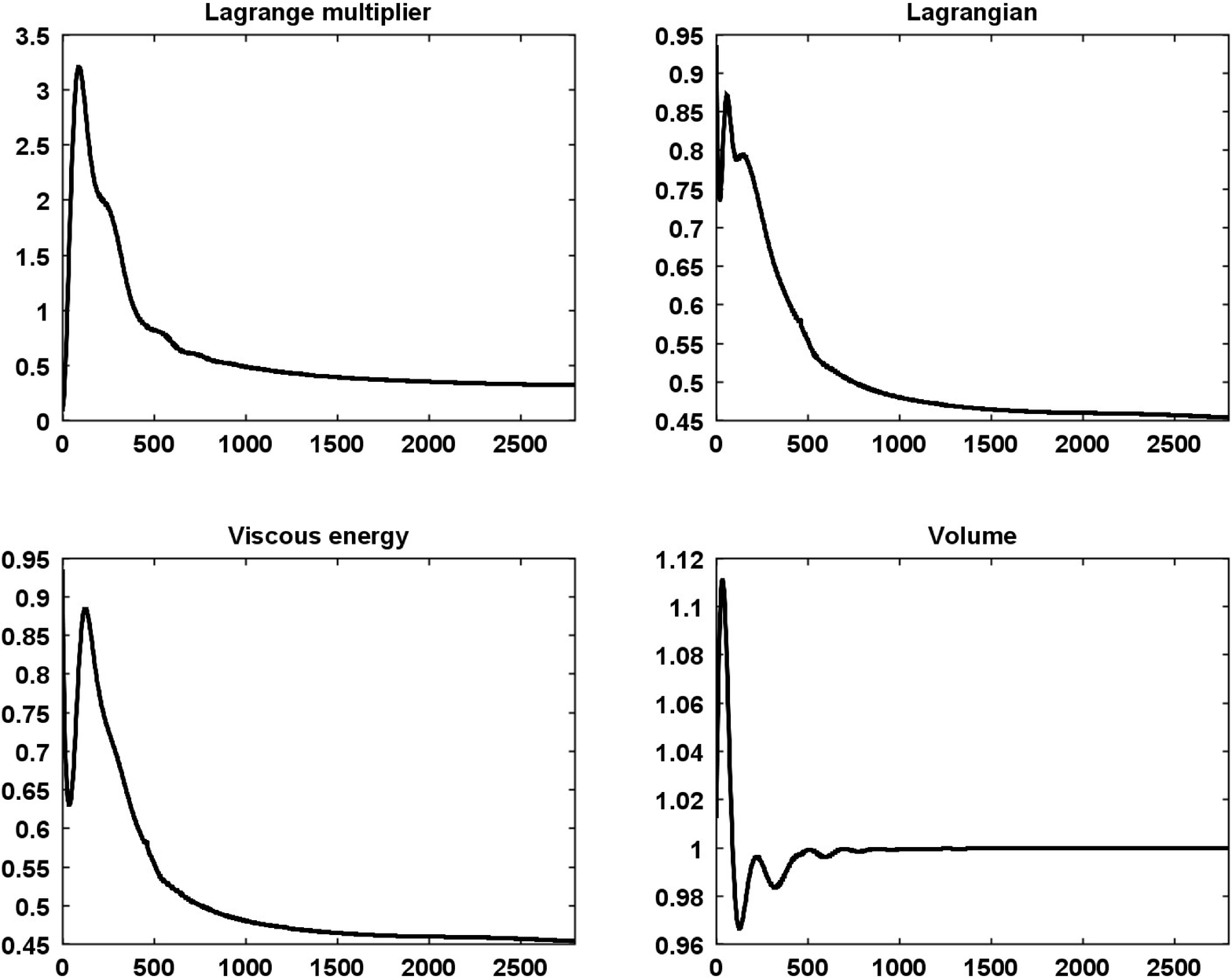}
\caption{2D case. Convergence curves of the Lagrange multiplier $\mu_k$, the Lagrangian $\mathcal{L}_b(\Omega_k,\mu_k)$, the viscous energy $J(\Omega_k)$ and the volume $V(\Omega_k)$. The x-coordinates represent the iterations number.}
\label{numconv22D}
\end{center}
\end{figure}

\begin{figure}[!h]
\begin{center}
\includegraphics[height=5cm]{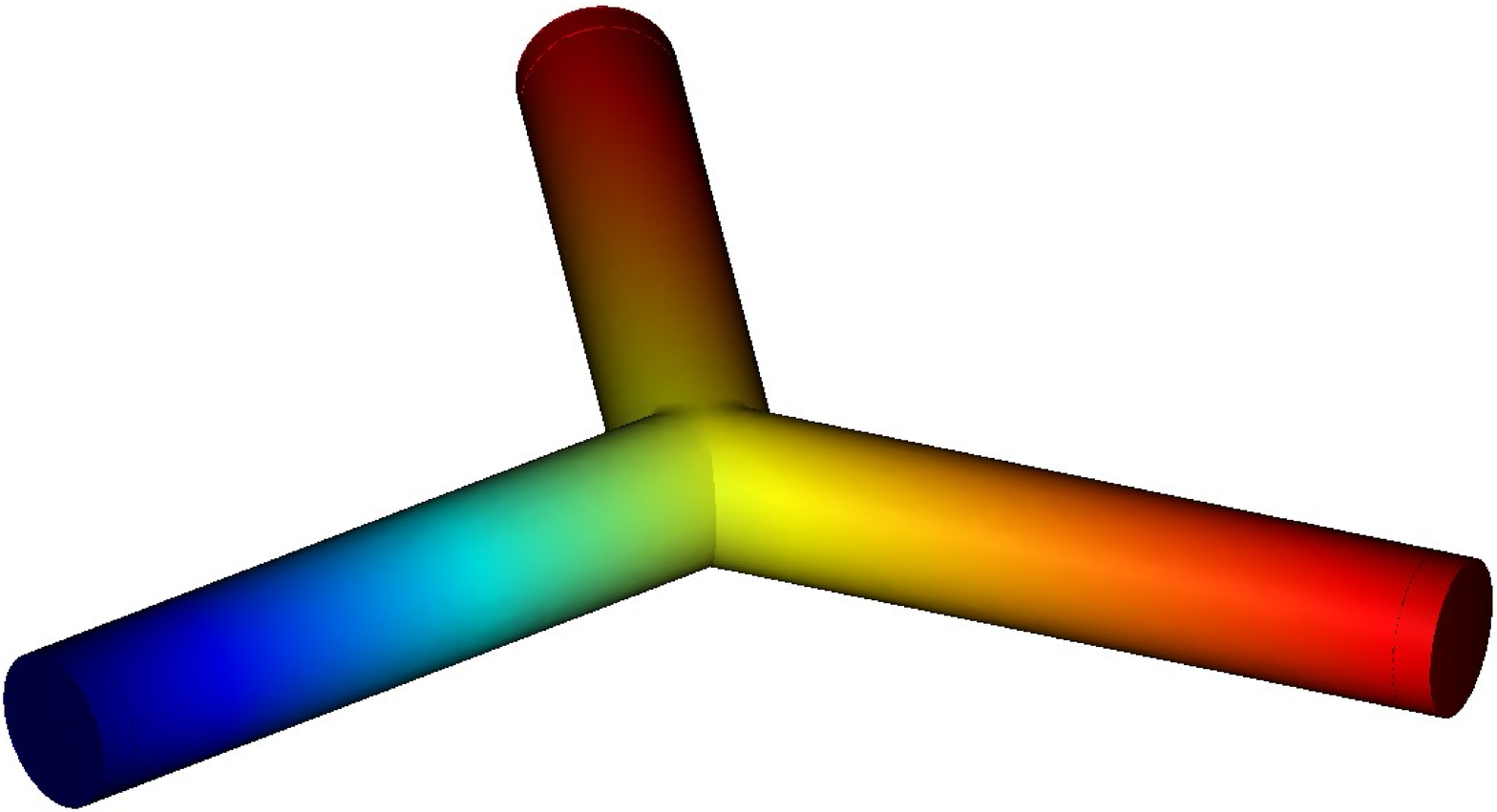}
\includegraphics[height=5cm]{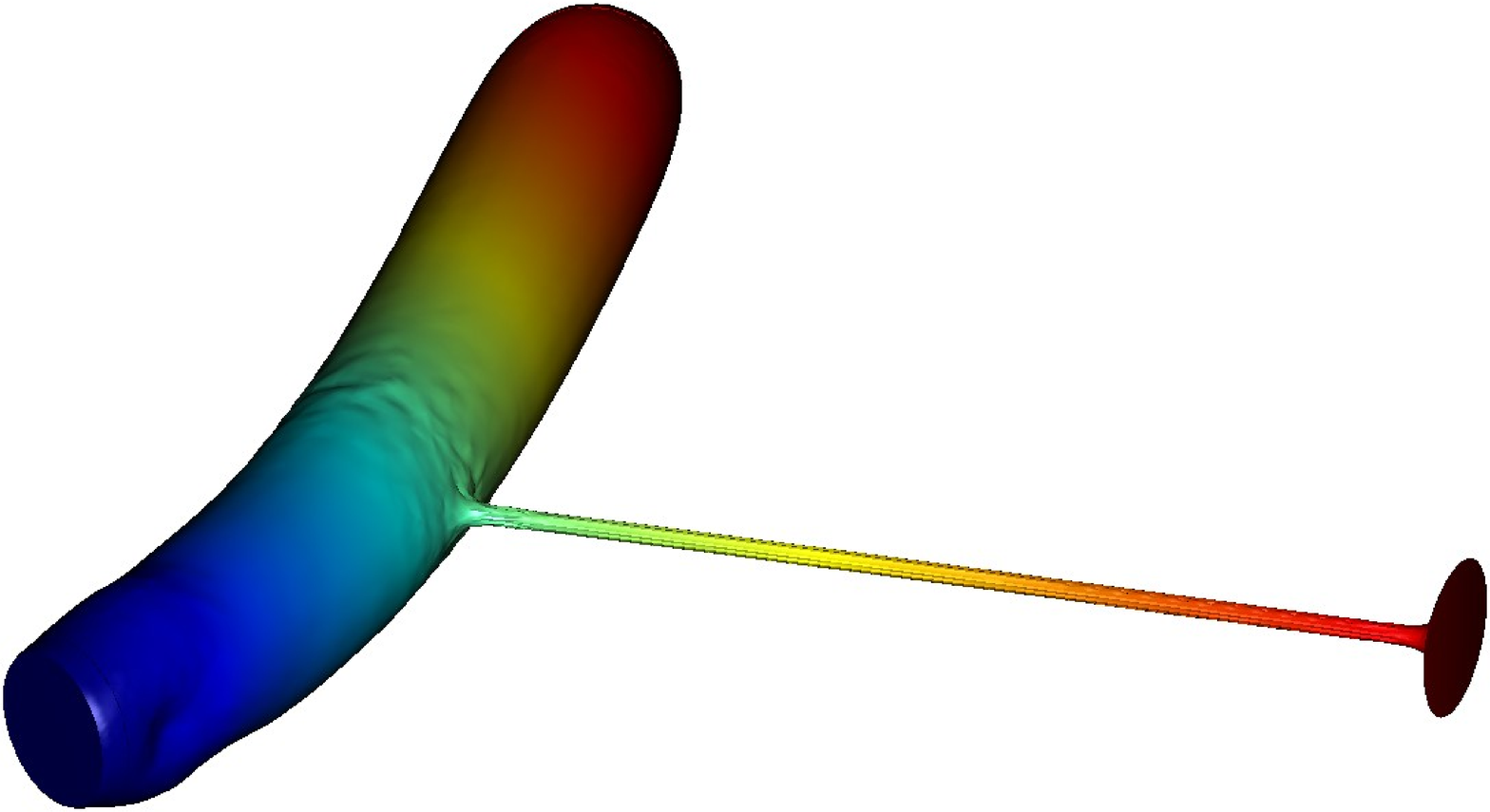}
\caption{Left: initial geometry $\Omega_0$, the inlet is the blue branch while the outlets are the red branches. Right: final shape reached by the algorithm.}
\label{numres2}
\end{center}
\end{figure}

\begin{figure}[!h]
\begin{center}
\includegraphics[height=8cm]{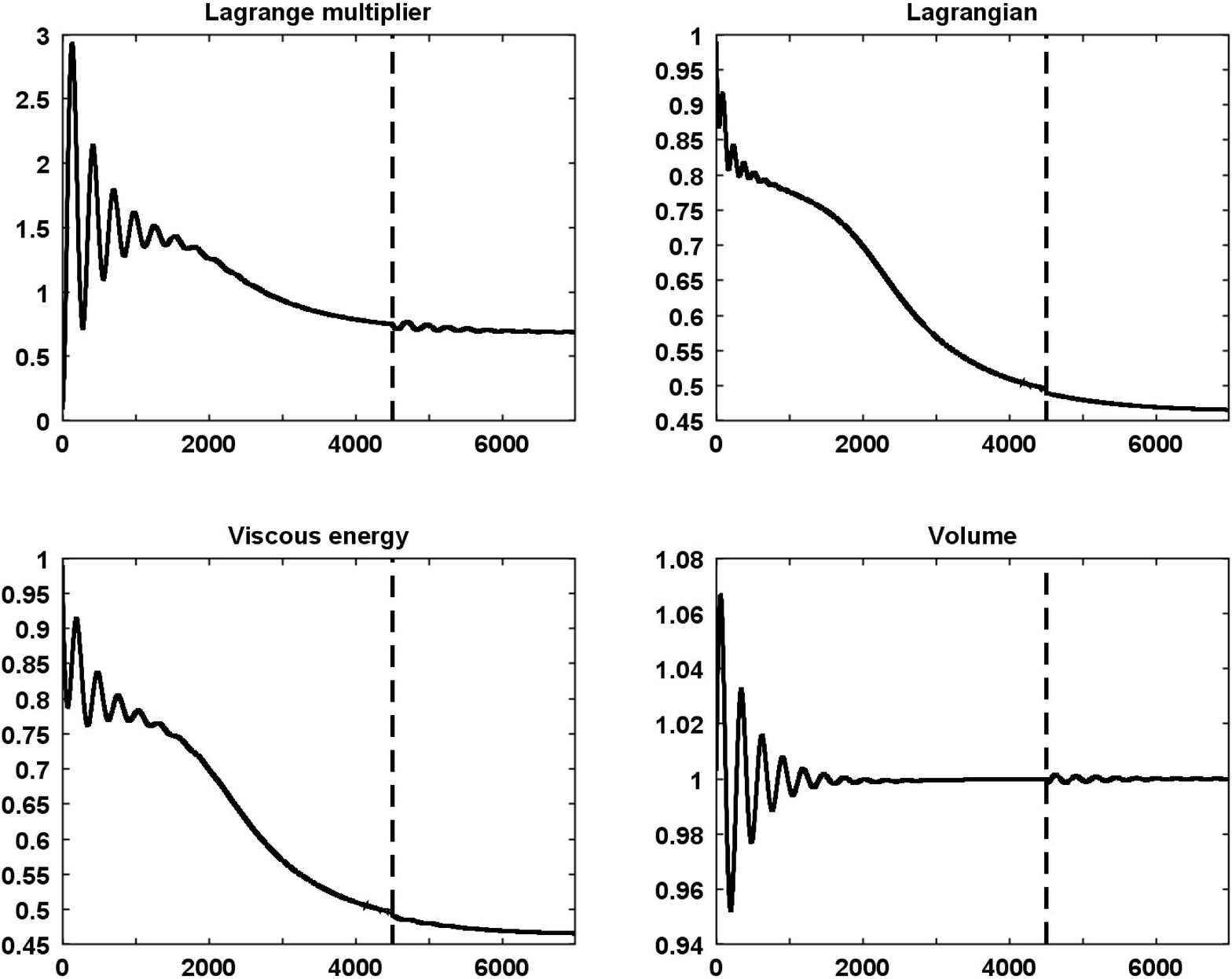}
\caption{Convergence curves of the Lagrange multiplier $\mu_k$, the Lagrangian $\mathcal{L}_b(\Omega_k,\mu_k)$, the viscous energy $J(\Omega_k)$ and the volume $V(\Omega_k)$. The x-coordinates represent the iterations number. The dashed line represents the step $4500$ at which a remeshing was done.} 
\label{numconv2}
\end{center}
\end{figure}

\section{Conclusion}

In this work, we show that fluid boundary conditions play an important role for the determination of the optimal shape of a tree in term of viscous dissipation. Indeed, the optimal shape associated to pressure conditions at leaves is a simple pipe while the optimal shape associated to flow conditions at leaves is a tree. Moreover we have shown that these results hold for both Poiseuille's regime and a regime with a Reynolds number $100$, which, although moderate, is large enough to exhibit inertial effects near the bifurcation.

Boundary conditions can be seen as constraints in the system and a pressure constraint is very different of a flow constraint. Let us consider a pipe in which flows a fluid in Poiseuille's regime. If we call $R$ the hydrodynamic resistance of the pipe, the dissipated energy is $J=R \Phi^2 = (\Delta p)^2 / R$ ($\Phi$ is the flow, $\Delta p$ is the pressure drop). Then minimizing $J$ relatively to the geometry of the pipe depends on the constraint:

\begin{description}
\item[Situation (i):] if the flow $\Phi$ is given then minimizing $J$ is equivalent to minimizing the resistance $R$ and the minimizer corresponds to $\Delta p=0$ (pipe radius goes to infinity).
\item[Situation (ii):] if the pressure drop $\Delta p$ is given then minimizing $J$ is equivalent to maximizing the resistance $R$ and the  minimizer corresponds to $\Phi = 0$ (no flow in the pipe, the pipe radius goes to $0$).
\end{description}

Simply speaking, the case of the tree we studied in this paper is an extrapolation of these two points depending on the conditions imposed at exits. When flows are imposed at outlets then each outlet is in Situation (i) and the optimal geometry is the one that tries to open the branches (in the limit of the constraint on the volume). When pressures are imposed, then all outlets are in Situation (ii) except one which is in Situation (i) since we imposed a non zero flow in the root of the tree and this flow, by conservation, has to get out of the tree by at least one branch. Moreover, in this last situation, our results show that it is ``better'' to close a branch and to use its volume to widen another branch. This implies that the other branch can distribute a larger amount of flow while dissipating less energy than two separate branches. Indeed, the viscous effects are larger near the walls and one wide branch has less walls than two smaller branches. When flow is imposed in a branch, this phenomena is compensated by the fact that reducing the branch radius increases the viscous effects by increasing the velocity gradients in the flow. Consequently, the optimal geometry is a compromise. These phenomena are probably true whenever the inertial effects are present or not, as our numerical simulations confirmed partially (Reynolds $100$). 



Finally, in term of modeling, our work shows that very different optimal structures can be obtained by boundary conditions adjustment (tree or pipe). For organs, we could imagine that depending on the organ function, optimization has been made through evolution by minimizing identical costs with different boundary conditions. 

\section*{Acknowledgments}
We thank the reviewer for his/her thorough review and highly appreciate the comments and suggestions, which significantly contributed to improve the quality of the publication.

\appendix

\section{Model proofs}

\label{app1}

\subsection{Proof of Proposition \ref{poiseuilleRelationProp}}

\label{app1-1}

The linearity of the relation between $\boldsymbol{p}$ and $\boldsymbol{q}$ comes from Poiseuille's law. Hence it is sufficient to compute the pressure vectors $\boldsymbol{p}$ related to the elements of canonical basis of $\mathbb{R}^{2^{N} \times 1}$. 
\par Let us begin with $\boldsymbol{q}=(1, 0, \ldots, 0)^\top  \in \mathbb{R}^{2^{N} \times 1}$. It corresponds to the case in which the fluid exits the tree only by the outlet denoted by $(N,1)$. In other terms, the fluid flows through the tree using the path $\Pi_{0 \to (N,1)}$. By conservation, the volumetric flow rate which enters the tree through the root branch is exactly $\Phi = 1$, so that the pressure $p_1$ at the outlet of the root branch is $p_0 - r_0$. As there is no flow in the right-hand subtree stemming from the root branch, the pressure at its outlets (whose couples of indexes are between $(N,2^{N-1} +1)$ and $(N,2^N)$) is the same as at the outlet of the root branch, namely $p_0 - r_0$. Similarly, the pressure $p_{1,1}$ at the outlet of the branch denoted by $(1,1)$ is $p_0 - (r_0 + r_{1,1}) = p_0 - r_0\left(1+\frac{1}{\xi_{1,1}}\right)$. Following this approach recursively, one finds pressures at the outlet of the branches of the path $\Pi_{0 \to (N,1)}$ denoted by $(i,1)$ with $i \in \llbracket 1,N \rrbracket$ to be 
$p_0 - \left(r_0 + 
\sum_{(k,l)\in \Pi_{0 \to (N,1)}(i)} r_{k,l}\right)$, which writes again $p_0 - r_0 \left(1+\sum_{(k,l) \in \Pi_{0 \to (N,1)}(i)} \frac{1}{\xi_{k,l}} \right)$. \par Therefore, the pressure $p_{N,j}$ (with $j \in \llbracket 1,2^N \rrbracket$) at the outlet of the tree is 
$$
p_{N,j}= p_0 - r_0 \left(1+\sum_{(k,l) \in \Pi_{0 \to (N,1)}(N-\nu_{0,j-1})} \frac{1}{\xi_{k,l}} \right).
$$
Applying the same reasoning to any vector $\boldsymbol{q} =  (0, \ldots, 0, 1, 0, \ldots, 0)^\top \in \mathbb{R}^{2^{N} \times 1}$ (with $1$ at the $i$-th position) yields
$$
\boldsymbol{p} = 
\left(p_0 - r_0 \left(1+\sum_{(k,l) \in \Pi_{0 \to (N,i)}(N-\nu_{i,0})} \frac{1}{\xi_{k,l}} \right), \ldots, 
p_0 - r_0 \left(1+\sum_{(k,l) \in \Pi_{0 \to (N,i)}(N-\nu_{i,2^N-1})} \frac{1}{\xi_{k,l}} \right)\right)^\top.$$
Consequently, we obtain
$$
\forall i \in \llbracket 1,2^N \rrbracket, \
p_{N,i} = p_0 - r_0 \sum_{j=1}^{2^N} q_j \left(1+\sum_{(k,l) \in \Pi_{0 \to (N,i)}(N-\nu_{i-1,j-1})} \frac{1}{\xi_{k,l}} \right)\\
$$

\subsection{Proof of Proposition \ref{defq}}

\label{app1-2}
The set of vectors $(\boldsymbol{v_i})_{i \in \llbracket 1, 2^N -1 \rrbracket}$ forms a basis of $\{ \boldsymbol{u_N}\}^{\perp}$, the orthogonal complement of $\boldsymbol{u_N}$. Let $i\in \llbracket 1,2^N-1\rrbracket$. Multiplying $(\ref{poiseuilleRelation})$ by $\boldsymbol{v_i}$ in the sense of the inner product yields
\begin{equation}
\forall i \in \llbracket 1,2^{N}-1 \rrbracket, \ 
- \langle \boldsymbol{p},\boldsymbol{v_i}\rangle = \langle A_N(\boldsymbol{\xi}) \boldsymbol{q}, \boldsymbol{v_i}\rangle.
\end{equation}
Furthermore, by conservation of the flow rate, one has $\langle \boldsymbol{q},\boldsymbol{u_N}\rangle = \Phi$.
\par Thus, since $A_N(\boldsymbol{\xi})$ is real symmetric, the vector $\boldsymbol{q}$ verifies the following system
\begin{equation}
\left\{
\begin{array}{l}
\langle \boldsymbol{q}, A_N(\boldsymbol{\xi}) \boldsymbol{v_i}\rangle = - \langle \boldsymbol{p},\boldsymbol{v_i}\rangle ,~\forall i \in \llbracket 1,2^{N}-1 \rrbracket \\
\langle \boldsymbol{q},\boldsymbol{u_N} \rangle = \Phi.\\
\end{array}
\right.
\end{equation}
\par That explains the structure of the matrix $M_N(\boldsymbol{\xi})$.
Let us now to prove that $M_N(\boldsymbol{\xi})$ is invertible. It amounts to prove that the set of vectors $(A_N(\boldsymbol{\xi})  \boldsymbol{v_1}, \ldots, A_N(\boldsymbol{\xi}) \boldsymbol{v_{2^N-1}}, \boldsymbol{u_N})$ is linearly independent. The fact that the set $(A_N(\boldsymbol{\xi})  \boldsymbol{v_1}, \ldots, A_N(\boldsymbol{\xi}) \boldsymbol{v_{2^N-1}})$ is linearly independent is almost trivial, since $A_N(\boldsymbol{\xi})$ is invertible by Proposition \ref{Ainvertible} and since the vectors $(\boldsymbol{v_i})_{i \in \llbracket 1, 2^N-1\rrbracket }$ are linearly independent.
\par Now, assume that
\begin{equation}\label{combiLin}
\lambda \boldsymbol{u_N} + \sum_{i=1}^{2^N-1} \mu_i A_N(\boldsymbol{\xi}) \boldsymbol{v_i} = 0 \textrm{ with } (\lambda,  \mu_{1}, \ldots, \mu_{2^N-1}) \in \mathbb{R}^{2^N}.
\end{equation}
\par Multiplying the equation (\ref{combiLin}) by the vector $A_N(\boldsymbol{\xi})^{-1}\boldsymbol{u_N}$ in the sense of the inner product yields
$$
\lambda \langle\boldsymbol{u_N},A_N(\boldsymbol{\xi})^{-1}\boldsymbol{u_N}\rangle + \sum_{i=1}^{2^N-1} \mu_i \langle A_N(\boldsymbol{\xi})\boldsymbol{v_i}, A_N(\boldsymbol{\xi})^{-1}\boldsymbol{u_N}\rangle = 0.
$$
Since $\boldsymbol{v_i} \in \{ \boldsymbol{u_N}\}^{\perp}$ for all $i \in \llbracket 1,2^{N}-1 \rrbracket$, it follows that
$$
\forall i \in \llbracket 1,2^{N}-1 \rrbracket, \
\langle A_N(\boldsymbol{\xi})\boldsymbol{v_i}, A_N(\boldsymbol{\xi})^{-1}\boldsymbol{u_N}\rangle
= \langle\boldsymbol{v_i}, \boldsymbol{u_N}\rangle= 0.
$$
Moreover, $\langle \boldsymbol{u_N},A_N(\boldsymbol{\xi})^{-1}\boldsymbol{u_N}\rangle >0 $. Indeed, we know that $A_N(\boldsymbol{\xi})$ is symmetric positive definite, and therefore $A_N(\boldsymbol{\xi})^{-1}$ inherits from these properties. Consequently, $\lambda=0$.
\par Now, (\ref{combiLin}) rewrites,
$$
\sum_{i=1}^{2^N-1} \mu_i A_N(\boldsymbol{\xi}) \boldsymbol{v_i} = 0.
$$
By the linear independence of the set $(A_N(\boldsymbol{\xi}) \boldsymbol{v_i})_{ i \in \llbracket 1, 2^N-1\rrbracket}$, we conclude that $
\forall i \in \llbracket 1,2^{N}-1 \rrbracket$, $\mu_i = 0$.
\par It proves that the vectors $(A_N(\boldsymbol{\xi})  \boldsymbol{v_1}, \ldots, A_N(\boldsymbol{\xi}) \boldsymbol{v_{2^N-1}}, \boldsymbol{u_N})$ are linearly independent.

\section{An augmented Lagrangian algorithm}\label{augLagAlg}

We present in this section the augmented Lagrangian algorithm used in this work to optimize the energy dissipated by a fluid in a dyadic tree with respect to the shape. 

One of the main interest of the augmented Lagrangian if compared to the classical Lagrangian is the regularization operation which generally improves the condition number of the dual function. 
Thus, a faster convergence of the maximizing sequence of the Lagrange multipliers is expected. 
However, the choice of a good parameter of augmentation can be very difficult. We will discuss this aspect concerning our simulations at the end of this section.
\par The descent direction in the main step of this algorithm will be computed thanks to a gradient method, which implies to calculate at each iteration the derivative of our criterion with respect to the domain. 
\par The augmented Lagrangian associated to Problem (\ref{pbOptim1}) is
\begin{equation}
\mathcal{L}_b(\Omega,\ell)=J(\Omega)+\ell G(\Omega)+\frac{b}{2}\left(G(\Omega)\right)^2,
\end{equation}
where $\ell\in \mathbb{R}$ is the Lagrange multiplier, associated with the volume constraint $G(\Omega)$, that is
$$
G(\Omega)=\textrm{meas}(\Omega)-V_0.
$$
It is then easy to determine the shape derivative of $\mathcal{L}_b$. One has
\begin{equation}
   \d \mathcal{L}_b (\Omega ; \boldsymbol{V})=\int_\Gamma
\left[2\mu \left(\varepsilon (\boldsymbol{u}):\varepsilon (\boldsymbol{v})-|\varepsilon
(\boldsymbol{u})|^2\right)+\ell+b(\textrm{meas}(\Omega)-V_0)\right](\boldsymbol{V}.\boldsymbol{n})\d s .
\end{equation}

We give now some precisions on the second step of the augmented Lagrangian algorithm, in particular on the choice of the descent method. Let $\Omega_k$ be the domain obtained at iteration $k-1$, $\Gamma_k$ its lateral boundary and $\mu_k$ the associated Lagrange multiplier. $\Omega_{k+1}$ is searched as a perturbation of the identity. That is why we write $\Omega_{k+1}=(I+\varepsilon_k\boldsymbol{d_k})(\Omega_k)$, where $\boldsymbol{d_k}$ is a vector field  representing the perturbation of the mesh and $\varepsilon_k$ a variable step.
\par Since we want to implement a gradient method, a first approach would be to consider $\boldsymbol{d_k}$ such that
\begin{equation}\label{choice1dk}
\boldsymbol{d_k}_{\mid_{\Gamma_k}}=-\nabla\mathcal{L}_b(\Omega_k,\ell_k), \ \forall k\in \mathbb{N}.
\end{equation}
This question has been much studied (see for instance \cite{allaire,burger,dogan,deGournay,mo-pi,protas}). In particular, in \cite{dogan}, the authors study similar methods applied to image segmentation, and exhibit some situations in which the choice of $\boldsymbol{d_k}$ as in (\ref{choice1dk}) is the worst solution from a numerical point of view. 
\par In this work, we chose $\boldsymbol{d_k}$ such that
$$
\Arrowvert \boldsymbol{d_k}\Arrowvert _{(H^1 (\Omega_k))^3}^2=-\langle \d \mathcal{L}_b(\Omega_k,\ell_k),\boldsymbol{d_k}),
$$
which corresponds to $\boldsymbol{d_k}$ solution of the equation
\begin{equation}\label{deplacement}
\left\{\begin{array}{ll}
-\Delta \boldsymbol{d_k} +\boldsymbol{d_k}=0 & x\in \Omega _k \\
\boldsymbol{d_k}=0 & x\in E\cup S\\
\displaystyle \frac{\partial \boldsymbol{d_k}}{\partial \boldsymbol{n}}=-\nabla \mathcal{L}_b(\Omega _k, \ell _k) & x\in \Gamma _k.
\end{array}
\right.
\end{equation}
We are now able to write the algorithm of resolution of the Problem (\ref{pbOptim1}).\\

\begin{center}
{\large \sl Augmented Lagrangian algorithm for the resolution of Problem (\ref{pbOptim1})}
\end{center}
\begin{enumerate}
\item \textit{Initialization.} Choose $\Omega _0 \in E$ and $\ell _0 \in\mathbb{R}$. \\
Let also fix $\tau >0$ and $\varepsilon _{\textnormal{stop}}$.
\item \label{iterk}\textit{Iteration $k$.} $\ell _k$ is known.
\begin{enumerate}
\item Resolution of the Navier-Stokes problem (and storage of its solution $\boldsymbol{u_k}$)
$$
\left\{\begin{array}{ll}
\displaystyle -\mu \Delta \boldsymbol{u_k}+\nabla p_k+\boldsymbol{u_k}\cdot\nabla \boldsymbol{u_k}=0 & x\in \Omega _{k}\\
\displaystyle \nabla\cdot \boldsymbol{u_k}=0 & x\in \Omega _k\\
\boldsymbol{u_k}=\boldsymbol{u_0} & x\in E \\
\boldsymbol{u_k}=0 & x\in \Gamma _k\\
-p_k\boldsymbol{n}+\mu \varepsilon (\boldsymbol{u_k})\boldsymbol{n}=-p_0.\boldsymbol{n} & x\in S.
\end{array}
\right.
$$
\item Resolution of the adjoint state (and storage of its solution $\boldsymbol{v_k}$)
$$
\left\{\begin{array}{ll}
\displaystyle -\mu \Delta \boldsymbol{v_k} + (\nabla \boldsymbol{u_k})^\top\!\! \ \boldsymbol{v_k} -(\nabla \boldsymbol{v_k}) \ \boldsymbol{u_k}+\nabla q _k=-2\mu \Delta \boldsymbol{u_k} & x\in \Omega_k\\
\nabla\cdot \boldsymbol{v_k}=0 & x\in \Omega _k\\
\boldsymbol{v_k}=0 & x\in E\cup \Gamma _k\\
-q_k\boldsymbol{n}+\mu \varepsilon (\boldsymbol{v_k})\boldsymbol{n}+(\boldsymbol{u_k}\cdot \boldsymbol{n})\boldsymbol{v_k}-4\mu \varepsilon (\boldsymbol{u_k})\boldsymbol{n}=0 & x\in S. 
\end{array}
\right.
$$
\item Calculation of the scalar
\begin{eqnarray*}
\beta _k & =  & \nabla \mathcal{L}_b(\Omega _k,\mu _k)\cdot \boldsymbol{n}\\
 & = & 2\mu\left( \varepsilon (\boldsymbol{u_k}):\varepsilon (\boldsymbol{v_k})-|\varepsilon (\boldsymbol{u_k})|^2\right)+\ell _k +b\left(\textnormal{meas }(\Omega _k)-V_0\right).
\end{eqnarray*}
\item Determination of the displacement of the mesh $\boldsymbol{d_k}$ as the solution of the elliptic equation
$$
\left\{\begin{array}{ll}
-\Delta \boldsymbol{d_k} +\boldsymbol{d_k}=0 & x\in \Omega _k \\
\boldsymbol{d_k}=0 & x\in E\cup S\\
\displaystyle \frac{\partial \boldsymbol{d_k}}{\partial \boldsymbol{n}}=-\beta _k\boldsymbol{n} & x\in \Gamma _k.
\end{array}
\right.
$$
\item Determination of an $\varepsilon _k$ that decreases the augmented Lagrangian.
\item Determination of the domain $\Omega _{k+1}$: $\Omega _{k+1}=(I+\varepsilon _k \boldsymbol{d_k})(\Omega _k)$.
\item Reinitialization of the Lagrange multiplier: $\ell_{k+1}=\ell_k+\tau  \left(\textnormal{meas }(\Omega _{k+1})-V_0\right)$.
\end{enumerate}\
\item \textit{Stopping criterion.} The algorithm stops if $\left| \ell _{k+1}-\ell_k\right| \leq \varepsilon _{\textnormal{stop}}$ and loops back to step \ref{iterk} if the inequality is false.
\end{enumerate}

\par As pointed out upwards, the choice of the parameter $b$ in the augmented Lagrangian algorithm can be difficult. Practically speaking, $b$ has to be chosen neither too big nor too small. Indeed, the biggest is the parameter $b$, the best is the conditioning of the dual functional giving the constraints and then the convergence of the sequence of Lagrange multipliers. Nevertheless, if $b$ is chosen too big, the conditioning of the primal problem $\min \{\mathcal{L}_b(\Omega, \ell_k),\Omega\in E\}$ deteriorates and it becomes more difficult to solve. Thus a compromise has to be done. Practically, a lot of preliminary tests have to be done to find $b$ before the algorithm could be run properly.



\begin{thebibliography}{00}
\bibitem{allaire} 
\newblock G. Allaire,
\newblock ``Conception optimale de structures'', Math\'ematiques \& Applications, {58}, Springer-Verlag, Berlin, 2007.

\bibitem{bejan} 
\newblock A. Bejan,
\newblock ``Shape and Structure, From Engineering to Nature'', Cambridge University Press, Cambridge, UK, 2000.

\bibitem{bernot} 
\newblock M. Bernot, V. Caselles, J.M. Morel, 
\newblock ``Optimal transportation networks: models and theory'', Lecture notes in mathematics (vol 1955), Springer, 2008.

\bibitem{burger} 
\newblock M. Burger,
\newblock {\it A framework for the construction of level set methods for shape optimization and reconstruction}, Interfaces and Free Boundaries, {\bf 5} (2003), 301--329.

\bibitem{BF} 
\newblock F. Boyer and P. Fabrie, 
\newblock ``El\'ements d'analyse pour l'\'etude de quelques mod\`eles d'\'ecoulements de fluides visqueux incompressibles'',  Math\'ematiques \& Applications, 52, Springer, Berlin, 2006.

\bibitem{bruneau} 
\newblock C.-H.~Bruneau, P.~Fabrie, 
\newblock {\it Effective downstream boundary conditions for incompressible Navier-Stokes equations}, {Int. J. for Num. Methods in Fluids}, {\bf 19} (1994), {8}, 693--705.

\bibitem{bruneau2} 
\newblock C.-H.~Bruneau, P.~Fabrie, 
\newblock {\it New efficient boundary conditions for incompressible {N}avier-{S}tokes equations: a well-posedness result}, 
\newblock {RAIRO Mod\'el. Math. Anal. Num\'er.}, {\bf 30}, {(1996)}, {7}, {815--840}.

\bibitem{Che1} 
\newblock D. Chenais, 
\newblock {\it On the existence of a solution in a
domain identification problem},
\newblock  J. Math. Anal. Appl., {\bf 52} (1975), 189--289.

\bibitem{DZ} 
\newblock M. Delfour and J.P. Zol{\'e}sio,
\newblock {``Shapes and Geometries. Analysis, Differential Calculus, and
Optimization''}, 
\newblock Advances in Design and Control SIAM, Philadelphia, PA,
2001.

\bibitem{dogan} 
\newblock G. Dog\u{g}an, P. Morin, R.H. Nochetto, M. Verani,
\newblock {\it Discrete Gradient Flows for Shape Optimization and Applications},
\newblock Comput. Methods Appl. Mech. Engrg. {\bf 196} (2007), no. 37-40, 3898--3914.

\bibitem{galdi}  
\newblock G. P. Galdi,
\newblock {``An Introduction to the Mathematical Theory of the
Navier-Stokes Equations''},
\newblock Volumes 1 and 2, Springer Tracts in Natural Philosophy , Vol.
38, 1998.

\bibitem{deGournay} 
\newblock F. de Gournay,
\newblock {\it Velocity extension for the level-set method and multiple eigenvalues in shape optimization},
\newblock SIAM J. Control Optim. {\bf 45} (2006), 343--367.

\bibitem{HP} 
\newblock A. Henrot and M. Pierre,
\newblock Variation et {``Optimisation de forme''},
\newblock Math\'ematiques et Applications, vol. 48, Springer 2005.

\bibitem{HPri} 
\newblock A. Henrot and Y. Privat,
\newblock {\it Une conduite cylindrique n'est pas optimale pour minimiser l'\'energie dissip\'ee par un fluide},
\newblock C. R. Math. Acad. Sci. Paris {\bf 346} (2008), no. 19-20, 1057--1061.

\bibitem{Henrot-Privat} 
\newblock A. Henrot and Y. Privat,
\newblock \textit{What is the optimal shape of a pipe?}, 
\newblock Arch. Ration. Mech. Anal. {\bf 196} (2010), no. 1, 281--302.

\bibitem{Hess} 
\newblock W.R. Hess, 
\newblock \textit{Das Prinzip des kleinsten Kraftverbrauchs im Dienste h\:amodynamischer Forschung}, Archiv. Anat. Physiol., 1914.

\bibitem{Heywood} 
\newblock J. Heywood, R. Rannacher and S. Turek, 
\newblock \textit{Artificial boundaries and flux and pressure conditions for the incompressible Navier-Stokes equations}, Internat. J. Numer. Methods Fluids, {\bf 22} (1996) 5.

\bibitem{mauroy-3DHydro} 
\newblock B. Mauroy, 
\newblock {\it 3D Hydronamics in the upper human bronchial tree: interplay between geometry and flow distribution}, {in ``Fractals in Biology and Medicine''}, IV, Birkhauser (2005).

\bibitem{mauroy-interplay} 
\newblock B. Mauroy, M. Filoche, J.S.~Andrade and B. Sapoval, 
\newblock {\it Interplay between geometry and flow distribution in an airway tree}, Physical Review Letters, {\bf 90} (2003), 1--4.

\bibitem{mauroy-optimalTreeDangerous} 
\newblock B. Mauroy, M. Filoche, E.R. Weibel and B. Sapoval, 
\newblock {\it An optimal bronchial tree may be dangerous}, Nature, {\bf 427} (2004), 633--636.

\bibitem{mauroy-meunier} 
\newblock B. Mauroy and N. Meunier,
\newblock {\it Optimal {P}oiseuille flow in a finite elastic dyadic tree}, M2AN Math. Model. Numer. Anal., {\bf 42} (2008), 4, 507--533.

\bibitem{mauryMeunier} 
\newblock B.~Maury, N.~Meunier, A.~Soualah and L.~Vial,
\newblock {\it Outlet dissipative conditions for air flow in the bronchial tree}, \newblock {in ``CEMRACS 2004---mathematics and applications to biology and medicine''}, {ESAIM Proc.}, 14, {201--212},{EDP Sci., Les Ulis} (2005).

\bibitem{mo-pi} 
\newblock B. Mohammadi and O. Pironneau,
\newblock ``Applied shape optimization for fluids'', Clarendon Press, Oxford 2001.

\bibitem{Mu-Si} 
\newblock F. Murat and J. Simon, 
\newblock {\it Sur le contr\^ole par un
domaine g\'eom\'etrique}, Publication du Laboratoire d'Analyse
Num\'erique de l'Universit\'e Paris 6, {\bf 189}, 1976.

\bibitem{Piro} 
\newblock O. Pironneau,
\newblock {``Optimal shape design for elliptic systems''},
\newblock Springer-Verlag, New York, 1984.

\bibitem{protas} 
\newblock B. Protas, T-R Bewley, and G. Hagen,
\newblock {\it A computational framework for the regularization of adjoint analysis in multiscale PDE systems},
\newblock J. Comput. Phys., {\bf 195} (2004), 49--89.

\bibitem{Quarteroni} 
\newblock A. Quarteroni and A. Veneziani, 
\newblock {\it Analysis of a geometrical multiscale model based on the coupling of ODEs and PDEs for blood flow simulations}, Multiscale Model. Simul., {\bf 1} (2)(2003).

\bibitem{Raux} 
\newblock M. Raux, M.N. Fiamma, T. Similowski and C. Straus, 
\newblock {\it Contr\^ole de la ventilation : physiologie et exploration en r\'eanimation}, R\'eanimation, {\bf 16} (2007).

\bibitem{simon1} 
\newblock J. Bello and E. Fern\'andez-Cara, 
\newblock {\it Optimal shape design for Navier-Stokes flow}, System modelling and optimization, P. Kall éd., Lecture Notes in Control and Inform. Sci., {\bf 180} (1992), 481--489.

\bibitem{simon2} 
\newblock J. Bello and E. Fern\'andez-Cara, 
\newblock {\it The variation of the drag with respect to the domain in Navier-Stokes flow}, Optimization, optimal control, partial differential equations, International Series of Numerical Mathematics, {\bf 107} (1992), 287--296.

\bibitem{So-Zo} 
\newblock J. Sokolowski and J. P. Zolesio, 
\newblock {``Introduction to Shape
Optimization: Shape Sensitivity Analysis''}, Springer Series in
Computational Mathematics, Vol. 16, Springer, Berlin 1992.

\bibitem{temam} 
\newblock R. Temam,
\newblock {``Navier-Stokes Equations''},
\newblock North-Holland Pub. Company, 1979.

\bibitem{tondeur} 
\newblock D. Tondeur and L. Luo, 
\newblock {\it Design and scaling laws of ramified fluid distributors by the constructal approach}, Chem.Eng.Sci. , {\bf 59} (2004), 1799--1813.

\bibitem{weibel} 
\newblock E.R. Weibel, 
\newblock {``The Pathway for Oxygen''}, Harvard University Press, Cambridge M A, 1984.
\end{thebibliography}
\end{document}